\RequirePackage{etex}
\documentclass[11pt]{amsart}
\usepackage{amsmath, amsthm, amssymb, amsfonts,color,url,enumitem, caption, booktabs}
\usepackage{subcaption}
\usepackage{thm-restate}
\usepackage{subcaption}
\usepackage{float} 
\usepackage{fullpage,ulem} 
\usepackage{amssymb} 
\usepackage{graphicx} 
\usepackage{hyperref} 
\usepackage{cleveref}
\usepackage{algpseudocode}
\usepackage{algorithm2e}
\usepackage{algpseudocode}
\usepackage{tikz}
\usepackage{tikz-3dplot}
\usepackage{diagbox}
\usetikzlibrary{positioning}
\usetikzlibrary{decorations}
\usetikzlibrary{decorations.text}
\usetikzlibrary{decorations.pathreplacing}
\usetikzlibrary{graphs,graphs.standard}
\tikzgraphsset{edges={draw,semithick}, nodes={circle,draw,semithick}}

\usepackage{lineno}

\usepackage{tikz-network}

\usepackage{comment}  
\usepackage{mathdots}
\usepackage{yhmath}
\usepackage{cancel}
\usepackage{color}
\usepackage{siunitx}
\usepackage{array}
\usepackage{multirow}
\usepackage{amssymb}
\usepackage{tabularx}
\usepackage{extarrows}
\usepackage{booktabs}
\usetikzlibrary{fadings}
\usetikzlibrary{patterns}
\usetikzlibrary{shadows.blur}
\usetikzlibrary{shapes}
\usepackage[colorinlistoftodos]{todonotes}

\newcommand{\Pin}{\mathsf{Pin}} 

\newcommand{\Poset}{\mathcal{P}} 

\newcommand{\pinn}{\mathsf{pin}} 

\theoremstyle{plain}
\newtheorem{theorem}{Theorem}[section]
\newtheorem{question}{Question}[section]
\newtheorem{corollary}{Corollary}[theorem]
\newtheorem{lemma}[theorem]{Lemma}
\newtheorem{proposition}[theorem]{Proposition}

\theoremstyle{definition}
\newtheorem{example}[theorem]{Example}
\newtheorem{definition}[theorem]{Definition}
\newtheorem{remark}[theorem]{Remark}

\hypersetup{colorlinks, citecolor=black, filecolor=black, linkcolor=black, urlcolor=blue}

\title{The Pinnacle Sets of a Graph}

\author[Bozeman]{Chassidy Bozeman}
\address[C. Bozeman]{Department of Mathematics and Statistics, Mount Holyoke College, 50 College St, South Hadley, MA 01075}
\email{
\textcolor{blue}{
\href{mailto:cbozeman@mtholyoke.edu}{cbozeman@mtholyoke.edu}
}
}

\author[Cheng]{Christine Cheng}
\address[C.~Cheng]{Department of Computer Science, University of Wisconsin-Milwaukee, WI 53211}
\email{
\textcolor{blue}{
\href{mailto:ccheng@uwm.edu}{ccheng@uwm.edu}
}
}

\author[Harris]{Pamela E. Harris}
\address[P.~E. Harris]{Department of Mathematical Sciences, University of Wisconsin-Milwaukee, WI 53211}
\email{
\textcolor{blue}{
\href{mailto:peharris@uwm.edu}{peharris@uwm.edu}
}
}

\author[S.~Lasinis]{Stephen Lasinis}
\address[S. Lasinis]{Department of Mathematical Sciences, University of Wisconsin-Milwaukee, WI 53211}
\email{
\textcolor{blue}{
\href{mailto:slasinis@uwm.edu}{slasinis@uwm.edu}
}
}

\author[Walker]{Shanise Walker}
\address[S. Walker]{Department of Mathematical Sciences, Clark Atlanta University, Atlanta, GA 30314}
\email{
\textcolor{blue}{
\href{mailto:swalker@cau.edu}{swalker@cau.edu}
}
}


\begin{document}

    \begin{abstract}
        We introduce and study the pinnacle sets of a simple graph $G$ with $n$ vertices. Given a bijective vertex labeling $\lambda\,:\,V(G)\rightarrow [n]$, 
        the label $\lambda(v)$ of vertex $v$ is a \textit{pinnacle} of $(G, \lambda)$ if $\lambda(v)>\lambda(w)$ for all vertices $w$ in the neighborhood of $v$. 
        The {\it pinnacle set of $(G, \lambda)$} contains all the pinnacles of the labeled graph. A subset $S\subseteq[n]$ is a {\it pinnacle set of $G$} if there exists a labeling $\lambda$ such that $S$ is the pinnacle set of $(G,\lambda)$.  Of interest to us is the question: {\it Which subsets of $[n]$ are the pinnacle sets of $G$?}
        
        Our main results are as follows. We show that when $G$ is connected, $G$ has a size-$k$ pinnacle set if and only if $G$ has an independent set of the same size.  Consequently, determining if $G$ has a size-$k$ pinnacle set and determining if $G$ has a particular subset $S$ as a pinnacle set are NP-complete problems.  Nonetheless, we completely identify all the pinnacle sets of complete graphs, complete bipartite graphs, cycles and paths. We also present two techniques for deriving new pinnacle sets from old ones that imply a typical graph has many pinnacle sets.  Finally, we define a poset on all the size-$k$ pinnacle sets of $G$ and show that it is a join semilattice.  If, additionally, the poset has a minimum element, then it is a distributive lattice. We conclude with some open problems for further study.
\end{abstract}

    \maketitle

\section{Introduction}

Graph labelings like proper colorings, graceful labelings, and distinguishing labelings have a long history in graph theory \cite{
albertson96distlabeling, bagga2015new, biatch2020survey, bloom1985graceful, 
cheng2006computing, 
edwards2006survey,
gallian2018dynamic,
golomb1972number,
robeva2011extensive, 
shivarajkumar2021graceful,
truszczynski1984graceful}.  In those settings, the vertices of a graph are labeled to achieve a certain goal.  In this paper, we study graph labelings to understand which labels can become prominent.  We refer to them as {\it pinnacles}.

Let $G$ be a graph with $n$ vertices and $[n] = \{1, 2, \hdots, n\}$.  A {\it (vertex) labeling} of $G$ is a function $\lambda: V(G) \rightarrow [n]$ that assigns distinct  labels to the vertices of $G$ from $[n]$. 
The label of vertex $u$, $\lambda(u)$,  is a {\it pinnacle} of  $(G, \lambda)$ if for any neighbor $v$ of $u$, $\lambda(u) > \lambda(v)$.  The {\it pinnacle set of $(G,\lambda)$} contains all the pinnacles of the labeled graph, and we denote this set by $\Pin(G,\lambda)$.
Figure~\ref{fig:pin example on graph} illustrates two distinct labelings of a graph with pinnacle sets $\{4,5\}$ and $\{5\}$. 
A graph has many labelings, and each one has their own pinnacle set.  We shall say that  $S \subseteq [n]$ is a {\it pinnacle set of $G$}  if there exists a labeling $\lambda$ of $G$ for which $S$ is the pinnacle set of $(G, \lambda)$. 
Of interest to us is the following question: {\it Given graph $G$, what are the pinnacle sets of $G$?}  

    \begin{figure}[h!]
    \centering
    \begin{tikzpicture}[roundnode/.style={circle, draw=black,  minimum size=.5mm}]
\node at (1.5,-1){$\Pin(G,\lambda_1)=\{4,5\}$};
\node[roundnode,fill=green!20](A1)at(0,0){$5$};
\node[roundnode,fill=white](A2)at(1.5,0){$2$};
\node[roundnode,fill=white](A3)at(3,0){$3$};
\node[roundnode,fill=white](A4)at(1.5,1.5){$1$};
\node[roundnode,fill=green!20](A5)at(3,1.5){$4$};

\draw (A1)--(A2)--(A3)--(A5)--(A4)--(A2)--(A5);
\end{tikzpicture}
\qquad\qquad\qquad
\begin{tikzpicture}[roundnode/.style={circle, draw=black,  minimum size=.5mm}]
\node at (1.5,-1){$\Pin(G,\lambda_2)=\{5\}$};
\node[roundnode](A1)at(0,0){$1$};
\node[roundnode,fill=white](A2)at(1.5,0){$3$};
\node[roundnode,fill=green!20](A3)at(3,0){$5$};
\node[roundnode,fill=white](A4)at(1.5,1.5){$2$};
\node[roundnode](A5)at(3,1.5){$4$};

\draw (A1)--(A2)--(A3)--(A5)--(A4)--(A2)--(A5);
\end{tikzpicture}
    \caption{Two graph labelings on $G$ with  pinnacles highlighted in green.} 
    \label{fig:pin example on graph}
\end{figure}
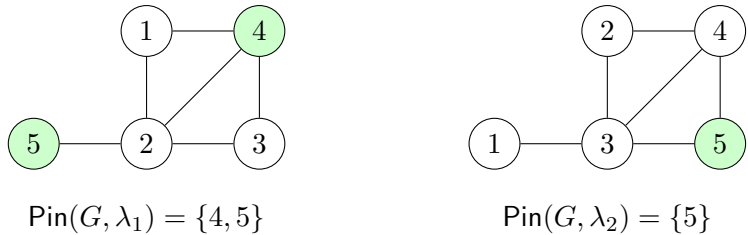

    Davis, Nelson, Petersen, and Tenner~\cite{davis2018pinnacle} were the first ones to introduce and study pinnacles in the context of permutations.
    Given a permutation of $[n]$ written in one line notation as $\pi=\pi_1\pi_2\cdots \pi_n$, the value $\pi_i$ is a \textit{pinnacle} of $\pi$ if $\pi_{i-1}<\pi_i>\pi_{i+1}$.  The {\it pinnacle set of $\pi$} contains all the pinnacles of $\pi$.  For example, when $\pi = 236154$, the pinnacle set of $\pi$ is $\{5,6\}$. Davis et al.'s work has now been extended to signed permutations and Stirling permutations in \cite{gonzalez2023pinnacle,mesas}. Additional enumerative results for pinnacle sets can be found in \cite{diaz2021formula,domagalski2022pinnacle,falque2021pinnacle,gonzalez2023pinnacle,rusu2021admissible}. We note that the study of pinnacles of permutations was motivated 
    by the study of \textit{peaks} of permutations, which are about the location of pinnacles and not their relative value. Much work was done in studying peaks of permutations, see \cite{billey2012permutations,castro2013number,DLsignedperms2017,peakpolyproof2017}, and Diaz-Lopez, Everham, Harris, Insko, Marcantonio, and Omar studied peaks on graphs in \cite{diaz2017peaks}.

    Every permutation $\pi$ of $[n]$ can be thought of as a labeling of the path $P_n$.  In our work,  we generalize the concept of pinnacles to graphs by replacing $P_n$ and $\pi$ with an arbitrary graph and a labeling of the graph.  We also make a small change -- in Davis et al.'s definition of pinnacles neither $\pi_1$ nor $\pi_n$ can be pinnacles because $\pi_0$ and $\pi_{n+1}$ do not exist; our definition has no such restrictions.  For example, the label of a degree-$0$ vertex is allowed to be a pinnacle because it satisfies the definition vacuously, and unlike in Davis et al.'s definition, our definition also allows pinnacles at leaves. As we shall see, the pinnacle sets of graphs have interesting connections with other concepts in graph theory.  They also behave nicely and have an orderly structure.
\bigskip

    \noindent {\it Our results.} 
    In Section 2, we present preliminary results about pinnacle sets of graphs.  First, we prove that {\it every} non-empty subset of integers with maximum element $n$ is the pinnacle set of some graph with $n$ vertices.  Second, we present a characterization of graphs that have size-$k$ pinnacle sets.  In particular, a connected graph $G$ has a pinnacle set of size $k$ if and only if $G$ has an independent set of the same size.  Third, the connection between pinnacle sets and independent sets allow us to prove that determining if a graph $G$ has a pinnacle set of size at least $k$ and determining if $G$ has a particular subset $S$ as a pinnacle set are both NP-complete problems.  Thus, two natural problems on pinnacle sets of graphs are computationally hard.

    In Section 3, we identify the pinnacle sets of complete graphs, complete bipartite graphs, cycles and paths.  In Section 4, we present two techniques for creating new pinnacle sets from old ones.  The first one involves switching the labels of two vertices.  Remarkably, by just performing a sequence of switches, we prove that if $P = \{p_1, p_2, \hdots, p_k\}$ with $p_1 < p_2 < \cdots< p_k$ is a pinnacle set of a graph $G$ with $n$ vertices and $Q = \{q_1, q_2, \hdots, q_k \}$ with $q_1 < q_2 < \cdots < q_k$ is a subset of $[n]$ such that $p_i \leq q_i$ for $i = 1, 2,\hdots, k$, then $Q$ is also a pinnacle set of $G$.   
    The second technique involves the use of what we call an {\it ordered tree partition} of a graph which allows us to switch the labels of two or more vertices.  
    We prove that if $G$ is connected and $P = \{p_1, p_2, \hdots, p_k\}$ with $p_1 < p_2 < \hdots< p_k$ is a pinnacle set of $G$, then $\{p_i, p_{i+1}, \hdots, p_k\}$ for $i = 2,3, \hdots, k$ are also pinnacle sets of $G$.  Both results suggest that an arbitrary graph will have many pinnacle sets. 

    In Section 5, we define a partial order (poset) on all the size-$k$ pinnacle sets of a graph.  We prove that the poset is a join semilattice, and if the poset has a minimum element, it is actually a distributive lattice.  We then present some scenarios when a graph has a minimum size-$k$ pinnacle set.  In such cases, the set of all size-$k$ pinnacle sets of the graph can be described succinctly.  We end in Section 6 with a conclusion and many suggestions for future work.

    \section{Pinnacles of a graph}\label{sec:pinancles of a graph}

In this section, we explore which subsets of $[n]$ can be pinnacle sets of a graph with $n$ vertices. We begin by establishing that every non-empty set of positive integers can arise as a set of pinnacles of a graph.

\begin{figure}
    \centering
    \begin{tikzpicture}[roundnode/.style={circle, draw=black, inner sep=0pt,  minimum size=5mm}]
    \draw[line width=0.1mm];
    \node[roundnode, fill=white](a1) at (0,0) {2};
    \node[roundnode, fill=white](a2) at (0,1) {3};
    \node[roundnode, fill=white](a3) at (0,2) {4};
    \node[roundnode, fill=green!20](a4) at (0,3) {5};
    \node[roundnode, fill=green!20] at (-1,0) {1};
    \node[roundnode, fill=green!20] at (1,0) {6};
    \node[roundnode, fill=white](b1) at (2,0) {7};
    \node[roundnode, fill=green!20](b2) at (2,1) {8};
    \draw (a1)--(a2)--(a3)--(a4);
    \draw (b1)--(b2);
\end{tikzpicture}
\qquad\qquad
    \begin{tikzpicture}[roundnode/.style={circle, draw=black, inner sep=0pt,  minimum size=5mm}]
    \draw[line width=0.1mm];
    \node[roundnode, fill=white](a1) at (0,0) {1};
    \node[roundnode, fill=green!20](c1) at (-1,.5) {3};
    \node[roundnode, fill=green!20](c2) at (-1,-.5) {5};
    \node[roundnode, fill=green!20](a2) at (1,0) {9};
    \node[roundnode, fill=white](a3) at (2,1.5) {2};
    \node[roundnode, fill=white](a4) at (2,.75) {4};
    \node[roundnode, fill=white](a5) at (2,0) {6};
    \node[roundnode, fill=white](a6) at (2,-.75) {7};
    \node[roundnode, fill=white](a7) at (2,-1.5) {8};
    \draw (a1)--(a2)--(a3);
    \draw (a1)--(c1);
       \draw (a1)--(c2);
          \draw (a2)--(a4);
          \draw (a2)--(a5);
          \draw (a2)--(a6);
          \draw (a2)--(a7);
\end{tikzpicture}

    \caption{The graphs above were created using the constructions described in the proof of Theorem \ref{thm:Construct graphs from pinnacle set}. The graph on the left has $S = \{1, 5, 6, 8\}$ as its pinnacle set. It is not connected because $1 \in S$ and $|S| \ge 2$.  The graph on the right has $S' = \{3, 5, 9\}$ as its pinnacle set.  The graph is a tree and this is possible because $1 \not \in S'$.}
    \label{fig:enter-label}
\end{figure}
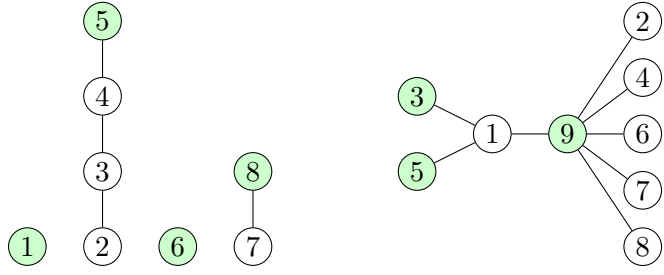

\begin{theorem}\label{thm:Construct graphs from pinnacle set}
    Let $S$ be a non-empty set of positive integers whose largest element is $n$. Then there is a graph $G$ with $n$ vertices so that $S$ is a pinnacle set of $G$. Furthermore, when $1 \not \in S$ or $|S| =1$,  $G$ can be a tree.  Otherwise, when $1 \in S$ and $|S| \ge 2$, $G$ cannot be a connected graph.  
\end{theorem}

\begin{proof}
    Let $S = \{s_1, s_2, \hdots, s_k\}$ such that $s_1 < s_2 < \cdots < s_k = n$. Let $s_0 = 0$.
    Create $k$ paths $G_1, G_2, \hdots, G_k$ so that $G_i$ has $s_i - s_{i-1}$ vertices.  Let $G$ be the resulting forest; it has $s_k = n$ vertices. For $i = 1, 2, \hdots, k$, moving from one end of $G_i$ to the other end, assign its vertices the labels $s_{i-1} + 1, s_{i-1} + 2, \hdots, s_i$. It is easy to check that $s_i$ is the only pinnacle in $G_i$ so $S$ is the pinnacle set of the labeled graph. We present an example in Figure~\ref{fig:enter-label}.
    
    When $1 \not \in S$ or $|S| =1$, we create a tree that has $S$ as its pinnacle set.  First, assume $1 \not \in S$ so $s_1 > 1$.
   Starting with two adjacent vertices $x$ and $y$,  attach $k-1$ leaves to $x$  and $n-k-1$ leaves to $y$. Call the resulting tree $T$; it has $2 + k-1 + n-k-1 = n$ vertices.  Assign $x$ the label 1 and its leaves the labels $s_1, s_2, \hdots, s_{k-1}$. Assign $y$ the label $n$ and its leaves the labels in $[n]\setminus (S \cup \{1\})$.  Again, it is easy to check that $s_1, s_2, \hdots, s_{k-1}$ and $s_k = n$ are pinnacles of the labeled graph while all the labels in $[n] - S$ are not.  See Figure~\ref{fig:enter-label} for an example.

    When $|S| = 1$, then $S=\{n\}$.  The case when $n > 1$ is addressed by the previous paragraph. We are left with the case when $n = 1$ so $S = \{1\}$.  Then let $T = K_1$ and label the single vertex with $1$.  Clearly, $\{1\}$ is the pinnacle set.  

    Finally, let us prove the last statement in the theorem by contradiction. Assume there is a connected graph $G$ and  a labeling $\lambda$ that has $S$ as its pinnacle set with $1 \in S$ and $|S| \ge 2$.  The latter condition on $S$ implies that $G$ has two or more vertices. Since $G$ is connected, the vertex with label $1$ has a neighbor, and this neighbor's label has to be larger than $1$.  That is, $1 \not \in S$, a contradiction. 
\end{proof}

Our next result establishes a connection between pinnacle sets and independent sets of a graph.
    \begin{proposition}\label{observe1}
         Let $G$ be a graph on $n$ vertices and $\lambda$ be a labeling of $G$.  Let $S$ be the pinnacle set of $(G, \lambda)$ and let $V_S$ contain all the vertices whose labels are in $S$. Then $n \in S$.  Moreover, $V_S$ is an independent set of $G$. 
    \end{proposition}

    \begin{proof}
         Let $z$ be the vertex such that $\lambda(z) = n$.  Since $n$ is the largest label assigned by $\lambda$, we have that $\lambda(z) > \lambda(y)$ for every neighbor $y$ of $z$.  Thus, $n \in S$.
         
         Next, when $V_S$ has only one vertex, $V_S$ is clearly an independent set.  
         So assume $V_S$ has two or more vertices.  
         Let $u, v \in V_S$.  
         Then either $\lambda(u) >  \lambda(v)$ or $\lambda(v) >  \lambda(u)$.  
         If $u$ and $v$ are adjacent, either $\lambda(u)$ or $\lambda(v)$ is not a pinnacle of $(G,\lambda)$, a contradiction.  
         Thus, no two vertices of $V_S$ are adjacent so $V_S$ is an independent set of $G$. 
    \end{proof}

\begin{definition}\label{def:basiclabeling} {\it Basic labeling.}  In our proofs on pinnacle sets, we  apply a certain type of labeling on a subgraph $G'$ of $G$ using labels from a set $L$ with $|L| \ge |V(G')|$ and an independent set $I$ of $G'$ so that every vertex of $G'$ is reachable from some vertex of $I$.   It works as follows: 
    First, partition  the vertices of $G'$ based on their distance from $I$. That is, set $D_0 = I$ and, for ${i \geq  1}$, let $D_i$ consist of the neighbors of the vertices in $D_{i-1}$ that are not in $\cup_{j=0}^{i-1} D_j$. At some point, there will be an index $d$ so that $D_d$ is not empty, but $D_{d+1}$ is empty, which implies that  $D_0 \cup D_1 \cup \cdots \cup D_d$  is a partition of the vertices of $G'$.    Then, for $i = 0$ to $d$, label the vertices in $D_i$ with the largest unused labels in $L$.   We shall call the resulting labeling a {\it basic labeling of $(G', L, I)$.}
\end{definition}  

Figure~\ref{basiclabelingexample} illustrates a basic labeling of the Petersen graph, where $I$ consists of the three green vertices and $L = [10]$.
We note that when some $D_i$ has two or more vertices, a basic labeling of $(G', L, I)$  is not unique since we do not specify the exact labeling of the vertices in $D_i$, we just specify that they get the largest unused labels in $L$.  Nonetheless, we can easily determine the pinnacle set of such a labeling.

    \begin{lemma}\label{basiclabelinglemma}
        Let $G'$ be a subgraph of $G$,  $L \subseteq [n]$ a set of labels such that $|L| \ge |V(G')|$ and $I$ an independent set  of $G'$ so that every vertex of $G'$ is reachable from some vertex of $I$.  Let  $\lambda$ be a basic labeling of $(G', L, I)$.  Then the pinnacle set of $(G', \lambda)$  consists of the largest $|I|$ labels in $L$.  
    \end{lemma}
    
    \begin{proof}
     
        From 
        Definition \ref{def:basiclabeling}, the basic labeling
        $\lambda$ assigned the vertices of $I$ the largest labels in $L$, the vertices of $D_1$ the next largest labels, etc.  
        That is,  as we move from $D_0$ to $D_1$ to $D_2$, etc., the vertices' labels decrease.  Notice that for any vertex $u \in \cup_{i=1}^d D_i$, $\lambda(u)$ cannot be a pinnacle of $(G', \lambda)$ because $u$ has a neighbor in $D_{i-1}$ that was assigned a larger label.  On the other hand, for a vertex $v \in I$, $\lambda(v)$ is a pinnacle of $(G', \lambda)$ because $I$ is an independent set so all the neighbors of $v$ are in $D_1$ and were assigned a smaller label than $v$.  Consequently, the pinnacle set of $(G', \lambda)$ is $\{ \lambda(v): v \in I\}$, which is exactly the set of $|I|$ largest labels in $L$.
    \end{proof}

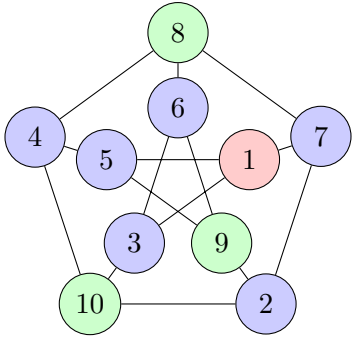
\begin{figure}[t]
    \centering        \begin{tikzpicture}[roundnode/.style={circle, draw=black,  minimum size=.8cm}]
\node[roundnode,fill=green!20](A1)at(0,2){8};
\node[roundnode,fill=blue!20](A2)at(1.9,0.62){7};
\node[roundnode,fill=blue!20](A3)at(1.17,-1.61){2};
\node[roundnode,fill=green!20](A4)at(-1.17,-1.61){10};
\node[roundnode,fill=blue!20](A5)at(-1.9,0.61){4};
\node[roundnode,fill=blue!20](B1)at(0,1){6};
\node[roundnode,fill=pink!80](B2)at(0.95,0.31){1};
\node[roundnode,fill=green!20](B3)at(0.58,-0.8){9};
\node[roundnode,fill=blue!20](B4)at(-0.58,-0.8){3};
\node[roundnode,fill=blue!20](B5)at(-0.95,0.31){5};
\draw(A1)--(A2)--(A3)--(A4)--(A5)--(A1);
\draw(B1)--(B3)--(B5)--(B2)--(B4)--(B1);
\draw(A1)--(B1);
\draw(A2)--(B2);
\draw(A3)--(B3);
\draw(A4)--(B4);
\draw(A5)--(B5);
\end{tikzpicture}
            \caption{A basic labeling of $(G, [10], I)$ where $G$ is the Petersen graph and $I$ consists of the green vertices (labeled 8, 9, 10). Then $D_1$ consists of the vertices in purple (labeled 2 through 7) and $D_2$ consists of the  vertex in pink (labeled 1).}
            \label{basiclabelingexample}
        \end{figure}

We now present a characterization of graphs with size-$k$ pinnacle sets. 
        
    \begin{theorem}\label{bigthm1}
        Let $G$ be a graph with $n$ vertices and $k \in [n]$.  The following statements are equivalent:

        \begin{enumerate}[leftmargin=.2in]
            \item[(1)] $G$ has a pinnacle set of size $k$.
            \item[(2)] The set $M_{n,k} = \{n-k+1, \ldots, n-1, n\}$ is a  pinnacle set of $G$. 
            \item[(3)] $G$ has an independent set $I$ of size $k$ and every vertex of $G$ is reachable from some vertex of $I$. 
        \end{enumerate}
    \end{theorem}
    
    \begin{proof}
        To prove the theorem, we  show that (1) $\Rightarrow$ (3) $\Rightarrow$ (2) $\Rightarrow$ (1). 
        
        \medskip

        \noindent (1) $\Rightarrow$ (3): Assume the pinnacle set of $(G, \lambda)$ is $S$ and $|S| = k$.  
        Let $V_S$ contain the vertices of $G$ whose labels are in $S$.  
        From  Proposition~\ref{observe1}, $V_S$ is an independent set of size $k$.   
        Now, consider an arbitrary vertex $v$. If $v \in V_S$, then $v$ has a path of length $0$ to itself.   
        So suppose $v \not \in V_S$. 
        Let $C$ be the connected component of $G$ that contains $v$. 
        Let $z$ be the vertex in $C$ so that $\lambda(z)$ is the largest label in $C$.  
        Hence, $\lambda(z)$ is a pinnacle of $(G, \lambda)$, so $z \in V_S$.  
        Furthermore, there is a path from $z$ to~$v$. 
        \smallskip
        
        \noindent (3) $\Rightarrow$ (2): Let $\lambda$ be a basic labeling of $(G, [n], I)$. 
        The $k$ largest labels in $[n]$ are $n-k+1, n-k+2, \ldots, n$.  By Lemma~\ref{basiclabelinglemma}, the pinnacle set of $(G, \lambda)$ is  $\{n-k+1, n-k+2, \ldots, n\}$. 
        
        \medskip
        
        \noindent (2) $\Rightarrow$ (1): When $\{n-k+1, \ldots, n-1, n\}$ is a pinnacle set of $G$, then $G$ has a pinnacle set of size $k$. 
    \end{proof}

    \begin{corollary}\label{bigthmcor}
        Let $G$ be a connected graph with $n$ vertices and $k \in [n]$.  The following statements are equivalent:
        \smallskip
        
        \noindent (1) $G$ has a pinnacle set of size $k$.
        
        \noindent (2) The subset $M_{n,k} = \{n-k+1, \ldots, n-1, n\}$ is a  pinnacle set of $G$. 
        
        \noindent (3) $G$ has an independent set  of size $k$. 
    \end{corollary}
    
    \begin{proof}
        When $G$ is connected,  (3) in Theorem~\ref{bigthm1} reduces to just $G$ having an independent set of size $k$.
    \end{proof}

    \begin{corollary}\label{bigthmcor2}
        Let $G$ be a graph with $n$ vertices and  $c$ connected components.  The smallest-sized pinnacle set of $G$ has $c$ elements.
    \end{corollary}

    \begin{proof}
        Since $G$ has $c$ connected components, $G$ has an independent set of size $c$ that can reach every vertex of $G$, so we can just pick a vertex from each component.  By Theorem~\ref{bigthm1}, $G$ has a pinnacle set of size $c$.  Now, consider any labeling $\lambda$ of $G$.   The highest label in each connected component of $G$ is a pinnacle of $(G, \lambda)$.  Hence, the pinnacle set of $(G, \lambda)$ has size at least $c$. It follows that the size of the smallest pinnacle set of $G$ is $c$. 
    \end{proof}
    
    \subsection{Computational complexity results}\label{subsec:complexity}

    The close connection between independent sets and pinnacle sets suggest that certain problems about pinnacle sets are computationally hard. Consider the following decision problems:

    \bigskip
    \fbox{\begin{minipage}{.8\textwidth}   
    
    \noindent {\bf PinnacleSetSize}
    \smallskip
    
    \noindent {\it Input:} A graph $G$ with $n$ vertices and an integer $k$, where $1 \leq k \leq n$.
    
    \noindent {\it Question:} Does $G$ have a  pinnacle set of size at least $k$?
    \end{minipage}}

    \bigskip
    
    \fbox{\begin{minipage}{.8\textwidth}
    
    \noindent {\bf PinnacleSetExistence}
    \smallskip
    
    \noindent {\it Input:} A graph $G$ with $n$ vertices and a subset $S$ of $[n]$. 
    
    \noindent {\it Question:} Is $S$ a pinnacle set of  $G$?
    \end{minipage}}

    \bigskip

    We establish that both problems are NP-complete.  Before we do so, consider the well-known NP-complete problem (see \cite{NPchapter} for a proof)  {\bf IndependentSet:}
    Given a graph $G$ with $n$ vertices and an integer $k \in [n]$, does $G$ have an independent set of size at least $k$?  We claim that the version below is still NP-complete. 
    
    \begin{proposition}
        {\bf IndependentSet} remains NP-complete even when the input graph $G$ is connected.
    \label{prop:IS}
    \end{proposition}

    \begin{proof}  
          {\bf IndependentSet} is already known to be in NP so we just have to show that the additional restriction in the proposition does not affect the computational complexity of the problem.
         
          Suppose $(G,k)$ is the input to {\bf IndependentSet}, but $G$ is not connected. 
          Add a new vertex $z$ and let $z$ be adjacent to all the vertices of $G$.  Call the new graph $H$.  Notice that $H$ is now connected.  Furthermore, we argue that $(G,k)$ is a yes-instance of {\bf Independent Set} if and only if $(H,k)$ is also a yes-instance of {\bf Independent Set}. When $k = 1$, the statement is true since both $G$ and $H$ have independent sets of size $1$.  When $k \ge 2$, the statement is also true because (i) every independent set of $G$ is also an independent set of $H$ and (ii) every independent set of $H$ with size at least two will never contain $z$ and is therefore an independent set of $G$.  Since constructing $H$ from $G$ takes linear time, we have shown that {\bf Independent Set} polynomially reduces to the version of the problem when the input graph is connected.
    \end{proof}

    \begin{theorem}\label{thm:complexity}
        {\bf PinnacleSetSize} and {\bf PinnacleSetExistence} are NP-complete problems.
    \end{theorem}

    \begin{proof} 
        Let us start with {\bf PinnacleSetSize}. Suppose $(G,k)$ is a yes-instance of the problem.  Then a subset $V'$ of $V(G)$ can serve as a polynomial-size certificate.  Check that $|V'| \ge k$, $V'$ is an independent set of $G$ and every vertex of $G$ is reachable from some vertex of $V'$.  These tasks can all be done in polynomial time.  By Theorem~\ref{bigthm1}, $V'$ will pass all these checks if and only if $G$ has a pinnacle set of size at least $k$. 
        
        Consider the input $(H,k)$ to the NP-complete problem {\bf IndependentSet}. Assume $H$ is connected. By Corollary~\ref{bigthmcor}, $(H,k)$ is a yes-instance of {\bf Independent Set} if and only if $(H,k)$ is also a yes-instance of {\bf PinnacleSetSize}. Thus, we have shown that the version of {\bf IndependentSet} whose input graph is connected polynomially reduces to  {\bf PinnacleSetSize}.  Since the former is NP-complete by Proposition \ref{prop:IS}, it follows that {\bf PinnacleSetSize} is also NP-complete.

        For {\bf PinnacleSetExistence}, a labeling $\lambda$ of $G$ serves as a polynomial-size certificate of a yes-instance $(G,S)$.  Check   that $S$ is the pinnacle set of $(G, \lambda)$.  If so, then $S$ is a pinnacle set of $G$.  Computing the pinnacle set of $(G, \lambda)$ can be done in polynomial time. 
        
        To prove that the problem is NP-complete,         
        we again consider a reduction from {\bf IndependentSet} whose input graph $H$ is connected. We note that $H$ has an independent set of size at least $k$ if and only if it has an independent set of size exactly $k$. By  Corollary~\ref{bigthmcor}, $H$ has an independent set of size  $k$ if and only if $M_{n,k} = \{n-k +1, n- k+2, \ldots, n\}$ is a pinnacle set of $H$.  That is, $(H,k)$ is a yes-instance of {\bf IndependentSet} if and only if $(H, M_{n,k})$ is a yes-instance of 
        {\bf PinnacleSetExistence}.  Consequently, {\bf PinnacleSetIdentity} is NP-complete.
    \end{proof}

\section{Pinnacle sets for certain graph families}\label{sec:graph families}

Although we have just shown that {\bf PinnacleSetExistence} is NP-complete, 
for certain graph families, we can still give a complete characterization of its pinnacle sets. 
We present such results for complete graphs, complete bipartite graphs, cycles, and paths.
Let $\Pin(G)$ contain all pinnacle sets of $G$, and as is customary in combinatorics, let the lower case font denote its size; i.e., $\pinn(G)=|\Pin(G)|$.

\subsection{Complete graphs}
We show that $\{n\}$ is the only pinnacle set of $K_n$, the complete graph on $n$ vertices.

\begin{lemma}\label{lem:kn}
If $n\geq 1$, then $\Pin(K_n)=\{n\}$ and $\pinn(K_n)=1$.
\end{lemma}
\begin{proof}
For any labeling of $K_n$, each vertex is adjacent to the vertex labeled $n$, and is therefore not a pinnacle. Thus, $\{n\}$ is the only pinnacle set for $K_n$ and $\pinn(K_n)=1$.   \end{proof}

 \subsection{Complete bipartite graphs}
We now characterize all of the pinnacle sets of $K_{m,n}$,  the complete bipartite graph on $m+n$ vertices.
In what follows, whenever $a<b$ are integers we let $[a,b]=\{a,a+1,\ldots,b\}$ and whenever $a=b$ we let $[a,b]=\{a\}$.

\begin{lemma}\label{lem:kmn}
If $m\geq n\geq 1$, then $\Pin(K_{m,n})=\{[k,m+n]: n+1\leq k\leq m+n\}$ and $\pinn(K_{m,n})=m$.
\end{lemma}
\begin{proof} 
Let $V=\{v_1,v_2,\ldots,v_m\}$ and $W=\{w_1,w_2,\ldots,w_n\}$ be the disjoint sets of vertices of $K_{m,n}$. We now label the vertices of $K_{m,n}$ so that $[n+k', n+m]$, where $1\leq k'\leq m$, is the pinnacle set. Let $\lambda(v_j) = j$ for $j = 1$ to $k'-1$ and let $\lambda(v_j) = n+j$ for $j = k'$ to $m$.  Assign the remaining  labels in $[k', n+k'-1]$ to the vertices of $W$.  Let $V' = \{v_1, v_2, \hdots, v_{k'-1}\}$ and $V'' = V-V'$.  Notice that for every $v' \in V', w \in W, v'' \in V''$, $\lambda(v') < \lambda(w) < \lambda(v'')$.  Thus, none of the labels in $[1, k'-1] \cup [k', n+k'-1]$ are pinnacles while {\it every} label in $[n+k', n+m]$ is a pinnacle. Since $1 \leq k' \leq m$, it follows that $[n+1, n+m], [n+2, n+m], \hdots, [n+m, n+m]$ are all pinnacle sets of $K_{m,n}$. 

For the reverse inclusion, suppose $S$ is a pinnacle set of $(K_{m,n},\lambda)$.  Without loss of generality, assume $v \in V$ so that $\lambda(v) = m+n$.  An immediate consequence is that for each $w \in W$, $\lambda(w)$ is not a pinnacle because $v$ and $w$ are adjacent. 
  Let $s$ be the smallest label in $S$ and $\lambda(v') = s$. If $v'\in W$, then $v'$ and $v$ would be adjacent and $\lambda(v)>\lambda(v')$, contradicting that $s$ is a pinnacle.  Hence $v'\in V$.
  It must also be the case that for every $w \in W$, $\lambda(w) < s$. 
  Pick $t \in [s,m+n]$.  
  Since $s \leq t$, the vertex $v''$ that has $\lambda(v'') = t$ is in $V$.  Since $v''$ is adjacent to only the vertices in $W$, $t$ is also a pinnacle.  It follows that {\it every} label in $[s,m+n]$ is a pinnacle; i.e., $S = [s,m+n]$.  Finally, by Corollary~\ref{bigthmcor2}, $|S| \leq m$ because the largest independent set of $K_{m,n}$ is $m$.  Hence, 
 $S\in\{[k,m+n]: n+1\leq k\leq m+n\}$, as desired.

The count for the number of distinct pinnacle sets follows directly from the construction above. 
\end{proof}

In the special case of $m=1$, the complete bipartite graph $K_{1,n-1}$ is a star graph with $n$ vertices. The previous result immediately implies the following. 

\begin{corollary}\label{lem:pinnacles of stars} 
If $n\geq 1$, then $\Pin(K_{1,n-1})=\{[k,n]:2\leq k\leq n\}$ and $\pinn(K_{1,n-1})=n-1$. 
\end{corollary}

\subsection{Cycle graphs}
Let $C_n$ denote the cycle graph on $n$ vertices.

\begin{lemma}
If $S$ is a pinnacle set of $C_n$, then $|S|\leq \lfloor\frac{n}{2}\rfloor$.
\end{lemma}
\begin{proof}
    By Corollary~\ref{bigthmcor}, the size of a pinnacle set of a graph is the size of an independent set of the graph. 
    The maximum size of an independence set of a cycle on $n$ vertices is exactly $\lfloor\frac{n}{2}\rfloor$.
\end{proof}

The following definition and lemma will play a key role in subsequent arguments.

\begin{definition}
    Let $S=\{s_1,s_2,\ldots,s_k\}\subseteq[n]$.  Define
$\ell(s_i)=|\{x\in [n]\setminus S: x< s_i\}|$, the number of elements in $[n]$ that are less than $s_i$ and are not in $S$. 
\end{definition}

\begin{lemma}\label{lemma:gapping}
Let $S=\{s_1,s_2,\ldots,s_k\}$ with $s_1< s_2<\cdots< s_k$ be a pinnacle set of $(G, \lambda)$.  Let $\lambda(v_j) = s_j$ for $j = 1, \hdots, k$, and let $G_j$ be the subgraph of $G$ induced by the vertices whose labels are less than or equal to $s_j$. If 
some $s_i$ has $\ell(s_i) \leq i$ and $\deg(v_j) \ge 2$ for $j = 1, \hdots, i$, then $G_i$ contains a cycle. 
\end{lemma}

\begin{proof}
By definition, $G_i$ has $s_i = \ell(s_i) + i$ vertices.  But $\ell(s_i)\leq i$, so $G_i$ has at most $i+i = 2i$ vertices.
Let $V_i= \{v_1, v_2, \hdots, v_i\}$ and $N[V_i]$ be the closed neighborhood of $V_i$. Since $s_1, s_2, \hdots, s_i$ are pinnacles of $(G, \lambda)$, every vertex in $N[V_i]$ is in $G_i$ and, consequently, $\deg_G(v) = \deg_{G_i}(v)$ for each $v \in V_i$.  But $V_i$ remains an independent set of $G_i$ so it follows that $G_i$ has at least $\sum_{j=1}^i \deg(v_j) \ge 2i$ edges.  Since the number of edges of $G_i$ is at least the number of its vertices, $G_i$ has to have a cycle. \end{proof}

In Theorems \ref{thm:Pinnacles of Cn} and \ref{thm:pinnacles of paths}, we fully characterize the pinnacles sets of cycles and paths, respectively. To determine if a set $S=\{s_1,s_2,\ldots,s_k\},$ is a pinnacle set of $G$, where $G$ is $C_n$ or $P_n$, for each $1\leq i\leq k-1$, we consider the number of labels less than $s_i$ that are not in $S$. 
If there exists an index $i\in[k]$ for which there are too few such labels (as described in the theorems), we consider the subgraph induced on $N[V_i]$, the closed neighborhood of the set of vertices whose labels are $s_1,s_2,\ldots,s_i$. 
By giving bounds on the number of vertices and edges of this subgraph, we show by way of contradiction that $G$ contains a proper subgraph that contains a cycle (in Theorem \ref{thm:Pinnacles of Cn}) or that $G$ contains a proper subgraph that is all of $G$  (in Theorem \ref{thm:pinnacles of paths}).  

\begin{theorem}\label{thm:Pinnacles of Cn}
Let $n \geq  3$ and $k$ be positive integers with $ k \leq \lfloor\frac{n}{2}\rfloor$.  Let $S=\{s_1,s_2,\ldots,s_k\}$ be a subset of $[n]$ with $s_1<s_2<\cdots<s_k$. 
Then $S$ is a pinnacle set of $C_n$ if and only if $\ell(s_i)\geq i+1$ for all $1\leq i\leq k-1$ and $s_k = n$.
\end{theorem}

\begin{proof}

($\Rightarrow$) Suppose $S$ is a pinnacle set of $(C_n, \lambda)$.  Let $\lambda(v_j) = s_j$ for $j = 1, \hdots, k$.  Assume that for some $1 \leq i \leq k-1$, $s_i$ has $\ell(s_i)\leq i$, and let $G_i$ be the subgraph of $C_n$ induced by the vertices whose labels are less than or equal to $s_i$. Since every vertex in $C_n$ has degree $2$,
by Lemma ~\ref{lemma:gapping} $G_i$ has a cycle.  But $i < k$ so $G_i$ is a proper subgraph of $C_n$ and cannot contain a cycle.  Hence, it must be the case that for $1 \leq i \leq k-1$, $\ell(s_i)\geq i+1$.  Finally, by Proposition \ref{observe1}, $n \in S$ and has to be the largest pinnacle so $s_k = n$.

($\Leftarrow$) Suppose that $S$ is a pinnacle set of $C_n$ and satisfies $\ell(s_i)\geq i+1$ for all $1\leq i\leq k-1$ and $s_k = n$. 
We now construct $\lambda$, a graph labeling of $C_n$, with pinnacle set $S$.
To do this,  denote the vertices of $C_n$ by $1$ through $n$ (clockwise) so $\lambda(i)$ is the label of vertex $i$.
Let $\lambda(n)=s_k=n$. Then label the even numbered vertices $\lambda(2j)=s_j$, for $1\leq j\leq k-1$.
Since $\ell(s_1)\geq 2$, there exist $x_0,x_1<s_1$ with which we can label $\lambda(1)=x_0$ and $\lambda(3)=x_1$. 
Since $\ell(s_i)\geq i+1$, for all $2\leq i\leq k-1$, we know that at every step we have a new small number $x_i<s_i$ with which we can label $\lambda(2i+1)=x_i$. Note that 
$\lambda$ assigns a pinnacle at the even-numbered vertices $2$ through $2(k-1)$ (and at vertex $n$), and assigns values smaller than those neighboring pinnacles at the odd-numbered vertices $1$ through $2k-1$. 
To finish the labeling, it suffices to use the remaining labels in $[n]\setminus(S\cup \{x_0,x_1,\ldots,x_{k-1}\})$ on the vertices numbered $2k$ to $n-1$ placing those labels in increasing order.
This labeling of $C_n$ has pinnacle set $S$, as desired. 
For an illustration of this labeling, see Figure~\ref{fig:labelingCn}.
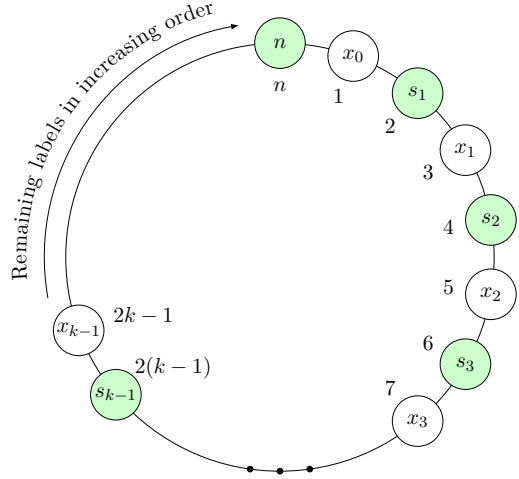
\begin{figure}
    \centering
 \resizebox{3in}{!}{
\begin{tikzpicture}[roundnode/.style={circle, draw=black, inner sep=0pt,  minimum size=.9cm}]
    \draw[line width=0.1mm] circle(1.5in);
    \node[roundnode, fill=green!20] at (90:1.5in) {\(n\)};
    \node at (90:1.2in) {\(n\)};

    \node[roundnode,fill=white] at (70:1.5in) {\(x_0\)};
    \node at (70:1.2in) {\(1\)};
    \node at (70:1.65in) {};

    \node[roundnode, fill=green!20] at (50:1.5in) {\(s_1\)};
    \node at (50:1.2in) {\(2\)};
    \node at (50:1.65in) {};

    \node[roundnode, fill=white] at (30:1.5in) {\(x_1\)};
    \node at (30:1.2in) {\(3\)};

    \node[roundnode, fill=green!20] at (10:1.5in) {\(s_2\)};
    \node at (10:1.2in) {\(4\)};

    \node[roundnode, fill=white] at (-10:1.5in) {\(x_2\)};
    \node at (-10:1.2in) {\(5\)};

    \node[roundnode, fill=green!20] at (-30:1.5in) {\(s_3\)};
    \node at (-30:1.2in) {\(6\)};

    \node[roundnode, fill=white] at (-50:1.5in) {\(x_3\)};
    \node at (-50:1.2in) {\(7\)};

    \node[roundnode, fill=black, minimum size=1mm] at (-82:1.5in) {};
    \node[roundnode, fill=black, minimum size=1mm] at (-90:1.5in) {};
    \node[roundnode, fill=black, minimum size=1mm] at (-98:1.5in) {};

    \node[roundnode, fill=green!20] at (-140:1.5in) {\(s_{k-1}\)};
    \node at (-140:1.2in) {\phantom{AAA}\(2(k-1)\)};

    \node[roundnode, fill=white] at (-160:1.5in) {\(x_{k-1}\)};
    \node at (-160:1.2in) {\phantom{AAA}\(2k-1\)};
    
    \node at (-160:1.2in) {};
    \node at (-160:1.65in) {\phantom{A}};

    \node(AS)at(-170:1.65in) {};
    \node(AT)at(-260:1.65in) {};
    \draw[-latex] (AS) arc[
        start angle=-170,
        end angle=-260,
        x radius=1.65in,
        y radius=1.65in
    ];
    \draw[color=white, rotate=-35, postaction={decorate, decoration={text along path, raise=4pt, text align={align=center}, text={Remaining labels in increasing order}, reverse path}}] (0,0) circle (1.72in);
\end{tikzpicture}
}
    \caption{The construction used in Theorem~\ref{thm:Pinnacles of Cn}.}
    \label{fig:labelingCn}
\end{figure}
\end{proof}

\begin{corollary}\label{cor:gaps for cycle}
Let $n \geq  3$ and $k$ be positive integers with $ k \leq \lfloor\frac{n}{2}\rfloor$.  Let $S=\{s_1,s_2,\ldots,s_k\}$ be a subset of $[n]$ with $s_1<s_2<\cdots<s_k$. Then $S$ is a pinnacle set of $C_n$ if and only $s_i\geq 2i+1$ for all $1\leq i\leq k-1$ and $s_k = n$.   
\end{corollary}

\begin{proof}
In the set $\{1,2,\ldots,s_i\}$, 
exactly $i$ of these labels are pinnacles and $s_i-1$ are not pinnacles. 
So $\ell(s_i)=s_i-i$. 
It follows from Theorem~\ref{thm:pinnacles of paths} that $S$ is a pinnacle set of $C_n$ if and only if $s_i-i\geq i+1$ for all $1\leq i\leq k-1$ and $s_k = n$. 
\end{proof}
In Table \ref{tab:num_pin_sets_cycle}
we give counts for the number of $k$-size pinnacle sets of the cycle on $n$ vertices. This array is known as the Catalan triangle and agrees with the OEIS entry \cite[\href{https://oeis.org/A008315}{A008315}]{OEIS}. We prove this is the case in Corollary \ref{cor:count pinnacles of cycle graph}
by establishing that the number of size $k$ pinnacles of the cycle graph on $n$ vertices is given by $\binom{n-2}{k-1}-\binom{n-2}{k-2}$.
Furthermore, in Corollary \ref{couting all pinnacle sets of Cn} we also show that the total number of distinct pinnacle sets of $C_n$ is given by $\binom{n-2}{\lfloor\frac n2\lfloor -1}$.

\begin{table}[ht]
    \centering
    \begin{tabular}{|c||c|c|c|c|c|}\hline
        \backslashbox{$n$}{$k$} & 1 & 2 & 3 & 4 & 5 \\ \hline\hline
        3  & 1 & 0 & 0 & 0 & 0 \\\hline
        4  & 1 & 1 & 0 & 0 & 0 \\\hline
        5  & 1 & 2 & 0 & 0 & 0 \\\hline
        6  & 1 & 3 & 2 & 0 & 0 \\\hline
        7  & 1 & 4 & 5 & 0 & 0 \\\hline
        8  & 1 & 5 & 9 & 5 & 0 \\\hline
        9  & 1 & 6 & 14 & 14 & 0 \\\hline
        10  & 1 & 7 & 20 & 28 & 14 \\\hline
        11  & 1 & 8 & 27 & 48 & 42 \\\hline
    \end{tabular}
    \caption{Number of $k$-size pinnacle sets of \(C_n\).}
    \label{tab:num_pin_sets_cycle}
\end{table}

\begin{remark}
    In Table \ref{tab:num_pin_sets_path} we 
    give the data for the number of $k$-size pinnacle sets for the path graph on $n$ vertices, which differs from Table \ref{tab:num_pin_sets_cycle} (for the number of $k$-size pinnacle sets of the cycle graph) by decrementing the parameter $n$ by one. 
\end{remark}

\subsection{Path graphs}
Let $P_n$ denote the path graph on $n$ vertices.

\begin{lemma}
If $S$ is a pinnacle set of $P_n$, then $|S|\leq \lceil\frac{n}{2}\rceil$.
\end{lemma}
\begin{proof}
    By Corollary~\ref{bigthmcor}, the size of a pinnacle set is the size of an independence set of the graph. 
    The maximal size of an independence set of a path on $n$ vertices is exactly $\lceil\frac{n}{2}\rceil$.
\end{proof}

\begin{theorem}\label{thm:pinnacles of paths}
Let $n \ge 2$ and $k$ be positive integers with $k \leq \lceil\frac{n}{2}\rceil$.  Let $S=\{s_1,s_2,\ldots,s_k\}$ with $s_1<s_2<\cdots<s_k$.
Then $S$ is a pinnacle set of $P_n$ if and only if $\ell(s_i)\geq i$ for all $1\leq i\leq k-1$ and $s_k = n$.
\end{theorem}
\begin{proof}
($\Rightarrow$)
Let $S$ be the pinnacle set of $(P_n, \lambda)$ and let $\lambda(v_j) = s_j$ for $j = 1, \hdots, k$. Assume there exists an $1\leq i\leq k-1$ for which $\ell(s_i)\leq i-1$. Let $G_i$ be the subgraph of $C_n$ induced by the vertices whose labels are less than or equal to $s_i$. Following the proof of Lemma \ref{lemma:gapping}, we note that $G_i$ has $s_i = i + \ell(s_i)$ vertices, and since $\ell(s_i)\leq i-1,$ $G_i$ has at most $i + i-1 = 2i-1$ vertices.  Furthermore, since $V_i= \{v_1, v_2, \hdots, v_i\}$ is an independent set, the number of edges of $G_i$ is at least $\sum_{j=1}^i \deg(v_j)$.  Now $V_i$ has at most two vertices with degree $1$ and the rest have degree $2$.  Thus, $\sum_{j=1}^i \deg(v_j) \ge 2(i-2) + 2 = 2i-2$. i.e., $G_i$ is a graph with at most $2i-1$ vertices and at least $2i-2$ edges.  

Recall that $G_i$ is subgraph of $P_n$ and has to be acyclic. For $G_i$ to have at most $2i-1$ vertices and at least $2i-2$ edges and be acyclic,  it has to have {\it exactly} $2i-1$ vertices and {\it exactly} $2i-2$ edges and it is connected. Having fewer than $2i-1$ vertices or having more than $2i-2$ edges or not being connected will immediately imply that $G_i$ contains a cycle, a contradiction. Thus, $G_i$ is a path that has the two endpoints of $P_n$ since two of its vertices have degree $1$.  That is,  $G_i$ itself is $P_n$. But $i < k$ so $G_i$ is a proper subgraph of $P_i.$ 
  Consequently, $\ell(s_i)\geq i$ for all $1\leq i\leq k-1$.  That $s_k = n$ follows directly from Proposition \ref{observe1}.

($\Leftarrow$) This proof is similar to that of Theorem~\ref{thm:Pinnacles of Cn}. We aim to construct a labeling $\lambda$ of $P_n$ with $S$ as pinnacle set. 
We begin by denoting the vertices of $P_n$ as $1$ through $n$ (left to right) so  $\lambda(i)$ is the label of vertex $i$.
Let $\lambda(n)=s_k=n$. 
Then label the odd-numbered vertices $\lambda(2j-1)=s_j$, for $1\leq j\leq k-1$. Since $\ell(s_1)\geq 1$, there exist $x_1<s_1$ with which we can label $\lambda(2)=x_1$. 
Since $\ell(s_i)\geq i$, for all $2\leq i\leq k-1$, we know that at every step we have a new small number $x_i<s_i$ with which we can label $\lambda(2i)=x_i$. Note that 
$\lambda$ places a pinnacle at the odd-numbered vertices $1$ through $2k-3$ (and at vertex $n$), and places values smaller than those neighboring pinnacles at the even-numbered vertices $1$ through $2(k-1)$. 
To finish the labeling, it suffices to use the remaining labels in $[n]\setminus(S\cup \{x_1,x_2,\ldots,x_{k-1}\})$ on the vertices numbered $2k-1$ to $n-1$ placing those labels in increasing order.
This labeling of $P_n$ has pinnacle set $S$, as desired. 
For an illustration of this labeling, see Figure~\ref{fig:labelingPn}.
\begin{figure}
    \centering
 \resizebox{5.5in}{!}{
\begin{tikzpicture}[roundnode/.style={circle, draw=black, inner sep=0pt,  minimum size=.9cm}]   
    \node[roundnode,fill=green!20](A1)at(0,0){$s_1$};
    \node(T1)at(0,-1){1};

    \node[roundnode](A2)at(2,0){$x_1$};
    \node(T2)at(2,-1){2};

    \node[roundnode,fill=green!20](A3)at(4,0){$s_2$};
    \node(T3)at(4,-1){3};

    \node[roundnode](A4)at(6,0){$x_2$};
    \node(T4)at(6,-1){4};
    
    \node(dots1) at(7,0){$\cdots$};
    \node[roundnode,fill=green!20](A5)at(8,0){$s_{k-1}$};
    \node(T5)at(8,-1){\(2k-3\)};
    
    \node(dots2) at(9,0){$\cdots$};
    
    \node[roundnode](A6)at(10,0){\(x_{k-1}\)};
    \node(T6)at(10,-1){\(2k-1\)};

    \node(BS)at(10.5,0.5){};
    \node(BT)at(16.5,0.5){};
    \draw [decoration={brace}, decorate] (BS) -- node [pos=0.5,style={above=3pt}] {\small{Remaining labels in increasing order}} (BT);

\node(dots3)at (13.5,0){$\cdots$};
    
    \node[roundnode,fill=green!20](A7)at(17,0){\(n\)};
    \node(T7)at(17,-1){\(n\)};
    
    \draw(A1)--(A2)--(A3)--(A4)--(dots1)--(A5)--(dots2)--(A6)--(dots3)--(A7);
\end{tikzpicture}
}

    \caption{The construction used in Theorem~\ref{thm:pinnacles of paths}.}
    \label{fig:labelingPn}
\end{figure}
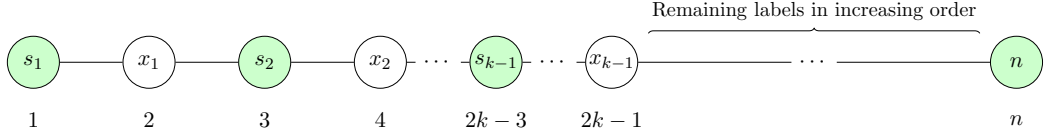
\end{proof}

\begin{corollary}
Let $n\geq 2$ and $k$ be positive integers with $k \leq \lceil\frac{n}{2}\rceil$. Let $S=\{s_1,s_2,\ldots,s_k\}$  be a subset of  $[n]$ with $s_1<s_2<\cdots<s_k$. Then $S$ is a pinnacle set of $P_n$ if and only $s_i\geq 2i$ for all $1\leq i\leq k-1$ and $s_k = n$.
\label{cor:gaps for paths}
\end{corollary}

\begin{proof}
Similar to the proof given in Corollary \ref{cor:gaps for cycle}, the set $\{1,2, \ldots,s_i\}$
contains exactly $i$ pinnacles and $s_i-1$ labels that are not pinnacles. So $\ell(s_i)=s_i-i$, and it follows from Theorem \ref{thm:pinnacles of paths} that $S$ is a pinnacle set of $P_n$ if and only if $s_i-i\geq i$ for all $1\leq i\leq k-1$ and $s_k = n$. 
\end{proof}

\begin{remark}
 As noted in the introduction, our work generalizes the notion of the pinnacle set of a permutation $\pi = \pi_1 \pi_2 \cdots \pi_n$ of $[n]$ in \cite{davis2018pinnacle}.  If we think of $\pi$ as a labeling of $P_n$, then our definition of pinnacle sets allows for the leaves to be pinnacles, while their definition does not.
 Consequently, in their work, the pinnacle set of $\pi$ need not contain $n$; it can even be an empty set. 
 In our work, however, the pinnacle set of {\it every} labeling of $P_n$ contains $n$. 

Davis et al.~\cite{davis2018pinnacle} characterized the non-empty subsets of positive integers that are pinnacle sets of permutations. We rephrase their statement below to match our discussion.

 \begin{theorem}[\cite{davis2018pinnacle}, Theorem 1.10 and Proposition 2.3]
 Let $S = \{s_1, s_2, \hdots, s_k\}$ be a non-empty set of positive integers with $s_1 < s_2 < \cdots < s_k$ and $s_k = n$.  Then $S$ is a pinnacle set of  a permutation of  $[n]$ if and only if both these statements are true:

    (i) $\{s_1, s_2, \hdots, s_{k-1} \}$ is also a pinnacle set of  
    a permutation of $[s_{k-1}]$  and
    
    (ii) $k < n/2$. 
 
\end{theorem}
 
 Note that Davis et al.'s characterization is recursive.  To prove that $S$ is a pinnacle set of a permutation of
  $[n]$, one has to verify that $k < s_k/2$ (remember that $s_k = n$), $k-1 < s_{k-1}/2$, $\ldots$, $1 < s_1/2$ and that $\{s_1\}$ is a pinnacle set of a permutation of 
  $[s_1]$, which is always true when $1 < s_1/2$.  Now, the $k$ inequalities can be summarized as $s_i > 2i$ for $i = 1, 2, \ldots, k$, 
 which bears a strong resemblance to the inequalities in Corollary \ref{cor:gaps for paths}.  Thus, Davis et al.'s characterization of the pinnacle sets of a permutation of
 $[n]$ is very similar but not the same as our characterization of the pinnacle sets of $P_n$ when $n \ge 2$.    
\end{remark}

We conclude this section with Table~\ref{tab:num_pin_sets_path}
where we give counts for the number of $k$-size pinnacle sets of the cycle on $n$ vertices. This array is known as the Catalan triangle and agrees with the OEIS entry \cite[\href{https://oeis.org/A008315}{A008315}]{OEIS}.
We prove this is the case in Corollary \ref{cor:count pinnacles of path graph} by establishing that the number of size $k$ pinnacles of the path graph on $n$ vertices is given by $\binom{n-1}{k-1}-\binom{n-1}{k-2}$. 
Furthermore, in Corollary \ref{couting all pinnacle sets of Pn} we also show that the total number of distinct pinnacle sets of $P_n$ is given by $\binom{n-1}{\lceil\frac n2\rceil -1}$.

\begin{table}[ht]
    \centering
    \begin{tabular}{|c||c|c|c|c|c|}\hline
        \backslashbox{$n$}{$k$} & 1 & 2 & 3 & 4 & 5 \\ \hline\hline
        2  & 1 & 0 & 0 & 0 & 0 \\\hline
        3  & 1 & 1 & 0 & 0 & 0 \\\hline
        4  & 1 & 2 & 0 & 0 & 0 \\\hline
        5  & 1 & 3 & 2 & 0 & 0 \\\hline
        6  & 1 & 4 & 5 & 0 & 0 \\\hline
        7  & 1 & 5 & 9 & 5 & 0 \\\hline
        8  & 1 & 6 & 14 & 14 & 0 \\\hline
        9  & 1 & 7 & 20 & 28 & 14 \\\hline
        10  & 1 & 8 & 27 & 48 & 42 \\\hline
    \end{tabular}
    \caption{Number of distinct \(k\)-size pinnacle set for $P_n$.}
    \label{tab:num_pin_sets_path}
\end{table}

\section{New pinnacle sets from old ones}\label{sec:new pinnacle sets from old ones}

We now present two methods for creating new pinnacle sets from old ones. The first one simply swaps the labels of two vertices while the second one relabels two or more vertices. 

\subsection{Switching labels}

Suppose that in the labeled graph $(G,\lambda)$, $\lambda(w) = p$ is a pinnacle while $\lambda(x) = p'$ is not a pinnacle and $p < p'$.  
This means that all the neighbors of $w$ have labels smaller than $p$ while $x$ has a neighbor whose label is larger than $p'$.  
Swap the labels of $w$ and $x$ and keep all the other vertices' labels the same.  
Call the new labeling $\lambda'$.  
Then,
$\lambda'(w) = p'$ is a pinnacle of $(G, \lambda')$ while $\lambda'(x) = p$ is not.  
It is also possible that the status of the other labels changed; i.e., they were not pinnacles in $(G, \lambda)$ but became pinnacles in $(G, \lambda')$ or vice versa. 
In the next lemma, we show that when $p' = p+1$, such status changes do not occur. For example, if $\{3, 6, 8, 10\}$ is a pinnacle set of $G$, then so are $\{4,6, 8, 10\}, \{3, 7, 8, 10\}$ and $\{3, 6, 9, 10\}$.

\begin{lemma}
\label{nextpinnacleset}
    Let $G$ be a graph  and $P= \{p_1, p_2, \ldots, p_k\}$ a pinnacle set of $G$ with $p_1 < p_2 < \cdots < p_k$.  Suppose that $p_{i+1} - p_{i} > 1$ for some  $i$, $1 \leq i < k$.  Then $(P \setminus \{p_i\}) \cup \{p_i + 1 \}$ is also a pinnacle set of $G$.  
\end{lemma}

\begin{proof} 
 Let $\lambda$ be a labeling of $G$ so that $P$ is the pinnacle set of $(G, \lambda)$.   Assume $\lambda(w) = p_i$ and $\lambda(x) = p_{i} + 1$. Define $\lambda'$ as the labeling that swaps the labels of $w$ and $x$  and keeps the labels of all other vertices fixed.  That is, $\lambda'(w) = p_i +1$, $\lambda'(x) = p_i$ and $\lambda'(v) = \lambda(v)$ for $v \neq w, x$.  
 Let $P'$ denote the pinnacle set of $(G,  \lambda')$.  
 We will now argue that $P' = (P \setminus \{p_i\}) \cup \{p_i + 1 \}$.

Notice that for any vertex $v$ that is not $w$ nor $x$ nor a neighbor of the two vertices,  the status of $\lambda(v) = \lambda'(v)$ in $(G, \lambda)$ and $(G,  \lambda')$ are the same.  That is,  $\lambda(v) \in P$ if and only if $\lambda'(v) \in P'$ because the labels of $v$ and that of its neighbors are the same in both $\lambda$ and $\lambda'$.  So let us consider $w$ and its neighbors first and then $x$ and its neighbors next.

Since $\lambda(w) = p_i \in P$ but $\lambda(x) = p_i+1 \not \in P$,  it must be the case that  $w$ and $x$ are not adjacent. It also means that all the neighbors of $w$ have $\lambda$-labels smaller than $p_i$ and are not in $P$.  These neighbors' labels did not change in $\lambda'$ but $\lambda'(w) > \lambda(w)$.  Thus, $\lambda'(w) = p_i +1 \in P'$ and the neighbors' $\lambda'$-labels are again not in $P'$.

On the other hand,  $\lambda(x) = p_i + 1 \not \in P$ because some neighbor of $x$ has a label larger than $p_{i}+ 1$. This neighbor's label did not change in $\lambda'$ but $\lambda'(x) < \lambda(x)$ so $\lambda'(x) = p_i  \not \in P'$.  
Next consider some neighbor $y$ of $x$ such that $\lambda(y) \in P$.   Then $\lambda(y)$ is larger than $p_i + 1$ while all of its neighbors have labels that are at most $p_i$.  But $\lambda'(y)  = \lambda(y)$ and all of its neighbors' $\lambda$-labels did not change in $\lambda'$ except for $x$ where $\lambda'(x) < \lambda(x)$.  (Note that $w$ is not a neighbor of $y$ since $\lambda(w) \in P$.) Thus, $\lambda'(y) \in P'$ too.

 Finally, suppose some neighbor $z$ of $x$ has $\lambda(z) \not \in P$,  then there are two possibilities:   (i) $\lambda(z) < \lambda(x) $ or (ii) $\lambda(z) >  \lambda(x)$ and $z$ has another neighbor whose label is larger than $\lambda(z)$.   For case (i), $\lambda(z) < p_i $  because $x$ and $w$ are not adjacent.  Thus, $\lambda(z) = \lambda'(z) < \lambda'(x)$ so $\lambda'(z) \not \in P'$.  For case (ii),   since $\lambda(x) = p_i + 1$, the $\lambda$-labels of $z$ and this neighbor are both larger than $p_i + 1$.  Both their labels did not change in $\lambda'$ so $\lambda'(z) \not \in P'$. 

We have shown that the only change from $P$ to $P'$ is that $p_i \not \in P'$ but $p_{i} +1 \in P'$.  Thus, $P'  = (P \setminus \{p_i\}) \cup \{p_i + 1 \}$. 
\end{proof}

Lemma~\ref{nextpinnacleset} leads us to an interesting fact about pinnacle sets. 

\begin{theorem}  
\label{bigthm2}
Let $G$ be a graph with $n$ vertices and $P = \{p_1, p_2, \ldots, p_k\} \subseteq [n]$ a pinnacle set of $G$ with $p_1 < p_2 < \cdots < p_k$.   Let $Q = \{q_1, q_2, \ldots, q_k \}$ be another subset of $[n]$ with $q_1 < q_2 < \cdots < q_k$.  When $p_i \leq q_i$ for all $1\leq i \leq k$, then $Q$ is a pinnacle set of $G$ as well.
\end{theorem}

\begin{proof}
Let $j$ be the largest index such that $p_j < q_j$.  It has to be the case that $1 \leq j < k$ because $p_k = q _k = n$ by Proposition~\ref{observe1}.
Thus,  $P = \{p_1, p_2, \ldots, p_j, q_{j+1}, q_{j+2}, \ldots, q_k \}$.   
Now $p_j < q_j < q_{j+1}$ so $q_{j+1} - p_j > 1$.  According to Lemma~\ref{nextpinnacleset}, it follows that 
$$ \{p_1, p_2, \ldots, p_{j-1},  p_j + 1, q_{j+1}, q_{j+2}, \ldots, q_k \}$$
is a pinnacle set of $G$. Applying this idea iteratively implies that the sets  
$$ \{p_1, p_2, \ldots, p_{j-1}, p_j + 2, q_{j+1}, q_{j+2}, \ldots, q_k \},$$
$$ \{p_1, p_2, \ldots,  p_{j-1}, p_j + 3, q_{j+1}, q_{j+2}, \ldots, q_k \},$$
$$ \vdots $$ 
$$ \{p_1, p_2, \ldots, p_{j-1}, q_j, q_{j+1}, q_{j+2}, \ldots, q_k \},$$
 are all pinnacle sets of $G$.  Repeating the procedure for the remaining indices, it follows that $Q$ is a pinnacle set of $G$. \end{proof}

\begin{figure}[ht!]
\begin{tabular}{cc}
        \centering
        \begin{subfigure}{.45\textwidth}
  \centering
  \begin{tikzpicture}[roundnode/.style={circle, draw=black,  minimum size=8mm}]
         
\node[]at(0,-2.5){};
\node[roundnode,fill=green!20](A1)at(0,2){7};
\node[roundnode,fill=white!20](A2)at(1.9,0.62){5};
\node[roundnode,fill=white!20](A3)at(1.17,-1.61){3};
\node[roundnode,fill=green!20](A4)at(-1.17,-1.61){10};
\node[roundnode,fill=white!20](A5)at(-1.9,0.61){6};
\node[roundnode,fill=white!20](B1)at(0,1){1};
\node[roundnode,fill=white!20](B2)at(0.95,0.31){8};
\node[roundnode,fill=green!20](B3)at(0.58,-0.8){4};
\node[roundnode,fill=white!20](B4)at(-0.58,-0.8){9};
\node[roundnode,fill=white!20](B5)at(-0.95,0.31){2};
\draw(A1)--(A2)--(A3)--(A4)--(A5)--(A1);
\draw(B1)--(B3)--(B5)--(B2)--(B4)--(B1);
\draw(A1)--(B1);
\draw(A2)--(B2);
\draw(A3)--(B3);
\draw(A4)--(B4);
\draw(A5)--(B5);
\end{tikzpicture}
  \caption{Pinnacle set $\{4,7,10\}$}
  \label{fig:sub1}
\end{subfigure}
         &
         \begin{subfigure}{.45\textwidth}
  \centering
  \begin{tikzpicture}[roundnode/.style={circle, draw=black,  minimum size=8mm}]
\node[roundnode,fill=green!20](A1)at(0,2){8};
\node[roundnode,fill=white!20](A2)at(1.9,0.62){5};
\node[roundnode,fill=white!20](A3)at(1.17,-1.61){3};
\node[roundnode,fill=green!20](A4)at(-1.17,-1.61){10};
\node[roundnode,fill=white!20](A5)at(-1.9,0.61){6};
\node[roundnode,fill=white!20](B1)at(0,1){1};
\node[roundnode,fill=white!20](B2)at(0.95,0.31){7};
\node[roundnode,fill=green!20](B3)at(0.58,-0.8){4};
\node[roundnode,fill=white!20](B4)at(-0.58,-0.8){9};
\node[roundnode,fill=white!20](B5)at(-0.95,0.31){2};\draw(A1)--(A2)--(A3)--(A4)--(A5)--(A1);
\draw(B1)--(B3)--(B5)--(B2)--(B4)--(B1);
\draw(A1)--(B1);
\draw(A2)--(B2);
\draw(A3)--(B3);
\draw(A4)--(B4);
\draw(A5)--(B5);
\end{tikzpicture}
  \caption{Pinnacle set $\{4,8,10\}$}
  \label{fig:sub2}
\end{subfigure}\\

   \begin{subfigure}{.45\textwidth}
  \centering
  \begin{tikzpicture}[roundnode/.style={circle, draw=black,  minimum size=8mm}]
\node[roundnode,fill=green!20](A1)at(0,2){9};
\node[roundnode,fill=white!20](A2)at(1.9,0.62){5};
\node[roundnode,fill=white!20](A3)at(1.17,-1.61){3};
\node[roundnode,fill=green!20](A4)at(-1.17,-1.61){10};
\node[roundnode,fill=white!20](A5)at(-1.9,0.61){6};
\node[roundnode,fill=white!20](B1)at(0,1){1};
\node[roundnode,fill=white!20](B2)at(0.95,0.31){7};
\node[roundnode,fill=green!20](B3)at(0.58,-0.8){4};
\node[roundnode,fill=white!20](B4)at(-0.58,-0.8){8};
\node[roundnode,fill=white!20](B5)at(-0.95,0.31){2};
\draw(A1)--(A2)--(A3)--(A4)--(A5)--(A1);
\draw(B1)--(B3)--(B5)--(B2)--(B4)--(B1);
\draw(A1)--(B1);
\draw(A2)--(B2);
\draw(A3)--(B3);
\draw(A4)--(B4);
\draw(A5)--(B5);
\end{tikzpicture}
  \caption{Pinnacle set $\{4,9,10\}$}
  \label{fig:sub3}
\end{subfigure}

         &
         
            \begin{subfigure}{.45\textwidth}
  \centering
  \begin{tikzpicture}[roundnode/.style={circle, draw=black,  minimum size=8mm}]
\node[roundnode,fill=green!20](A1)at(0,2){9};
\node[roundnode,fill=white!20](A2)at(1.9,0.62){4};
\node[roundnode,fill=white!20](A3)at(1.17,-1.61){3};
\node[roundnode,fill=green!20](A4)at(-1.17,-1.61){10};
\node[roundnode,fill=white!20](A5)at(-1.9,0.61){6};
\node[roundnode,fill=white!20](B1)at(0,1){1};
\node[roundnode,fill=white!20](B2)at(0.95,0.31){7};
\node[roundnode,fill=green!20](B3)at(0.58,-0.8){5};
\node[roundnode,fill=white!20](B4)at(-0.58,-0.8){8};
\node[roundnode,fill=white!20](B5)at(-0.95,0.31){2};
\draw(A1)--(A2)--(A3)--(A4)--(A5)--(A1);
\draw(B1)--(B3)--(B5)--(B2)--(B4)--(B1);
\draw(A1)--(B1);
\draw(A2)--(B2);
\draw(A3)--(B3);
\draw(A4)--(B4);
\draw(A5)--(B5);
\end{tikzpicture}
  \caption{Pinnacle set $\{5,9,10\}$}
  \label{fig:sub4}
\end{subfigure}
         \\
            \begin{subfigure}{.45\textwidth}
  \centering
  \begin{tikzpicture}[roundnode/.style={circle, draw=black,  minimum size=8mm}]
\node[roundnode,fill=green!20](A1)at(0,2){9};
\node[roundnode,fill=white!20](A2)at(1.9,0.62){4};
\node[roundnode,fill=white!20](A3)at(1.17,-1.61){3};
\node[roundnode,fill=green!20](A4)at(-1.17,-1.61){10};
\node[roundnode,fill=white!20](A5)at(-1.9,0.61){5};
\node[roundnode,fill=white!20](B1)at(0,1){1};
\node[roundnode,fill=white!20](B2)at(0.95,0.31){7};
\node[roundnode,fill=green!20](B3)at(0.58,-0.8){6};
\node[roundnode,fill=white!20](B4)at(-0.58,-0.8){8};
\node[roundnode,fill=white!20](B5)at(-0.95,0.31){2};
\draw(A1)--(A2)--(A3)--(A4)--(A5)--(A1);
\draw(B1)--(B3)--(B5)--(B2)--(B4)--(B1);
\draw(A1)--(B1);
\draw(A2)--(B2);
\draw(A3)--(B3);
\draw(A4)--(B4);
\draw(A5)--(B5);
\end{tikzpicture}
  \caption{Pinnacle set $\{6,9,10\}$}
  \label{fig:sub5}
\end{subfigure}
         &
         
             \begin{subfigure}{.45\textwidth}
  \centering
  \begin{tikzpicture}[roundnode/.style={circle, draw=black,  minimum size=8mm}]
\node[roundnode,fill=green!20](A1)at(0,2){9};
\node[roundnode,fill=white!20](A2)at(1.9,0.62){4};
\node[roundnode,fill=white!20](A3)at(1.17,-1.61){3};
\node[roundnode,fill=green!20](A4)at(-1.17,-1.61){10};
\node[roundnode,fill=white!20](A5)at(-1.9,0.61){5};
\node[roundnode,fill=white!20](B1)at(0,1){1};
\node[roundnode,fill=white!20](B2)at(0.95,0.31){6};
\node[roundnode,fill=green!20](B3)at(0.58,-0.8){7};
\node[roundnode,fill=white!20](B4)at(-0.58,-0.8){8};
\node[roundnode,fill=white!20](B5)at(-0.95,0.31){2};
\draw(A1)--(A2)--(A3)--(A4)--(A5)--(A1);
\draw(B1)--(B3)--(B5)--(B2)--(B4)--(B1);
\draw(A1)--(B1);
\draw(A2)--(B2);
\draw(A3)--(B3);
\draw(A4)--(B4);
\draw(A5)--(B5);
\end{tikzpicture}
  \caption{Pinnacle set $\{7,9,10\}$}
  \label{fig:sub6}
\end{subfigure}  
         \\
\end{tabular}
         \bigskip
         \caption{The figures above show how we transform a labeling of the Petersen graph that has $\{4, 7, 10\}$ as a pinnacle set to another labeling whose pinnacle set is $\{7, 9, 10\}$ according to the proof of Theorem~\ref{bigthm2}. } 
        \label{fig:petersengraphexample}
    \end{figure}
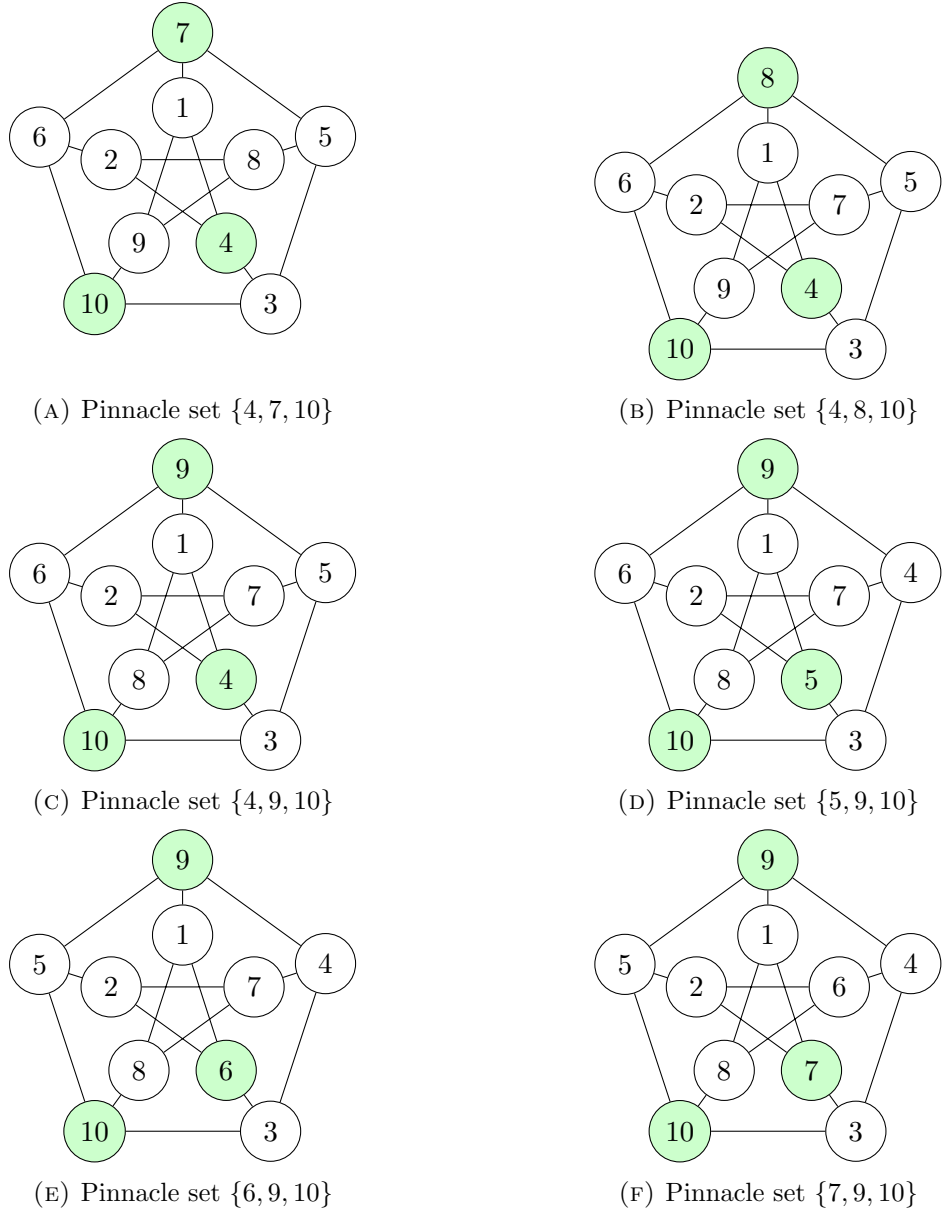

\bigskip

\begin{example}
    Let $G$ be the Petersen graph. Consider the labeling $\lambda$ of $G$ in Figure~\ref{fig:sub1}.  It is easy to check that $\{4, 7, 10 \}$ is the pinnacle set.  
    Now, $4 \leq 7$ and $7 \leq 9$ and $10 \leq 10$,  so according to Theorem~\ref{bigthm2},  $\{7, 9, 10\}$ is also a pinnacle set of the Petersen graph.  In the remaining subfigures of Figure \ref{fig:petersengraphexample}, we demonstrate the series of transformations that takes $\lambda$  to  a labeling that has $\{7, 9, 10\}$ as the pinnacle set based on the proof of Theorem~\ref{bigthm2}. 

The transformation can be summarized in two steps:  from $\{ 4, 7, 10 \}$ to $\{ 4, 9, 10 \}$ and from $\{ 4, 9, 10 \}$ to $\{7, 9, 10\}$.  Let $\lambda'$ and $\lambda''$ denote the labelings that have $\{ 4, 9, 10 \}$ and $\{7, 9, 10\}$ as pinnacle sets respectively.  The transformation from $\lambda$ to $\lambda'$ involves the vertices labeled $7, 8, 9$ in $(G, \lambda)$.  Their labels are circularly shifted in $(G, \lambda')$.  That is, the vertex labeled $7, 8, 9$ have labels $9, 7, 8$ in $(G, \lambda')$. 
Similarly, the transformation from $\lambda'$ to $\lambda''$ involves the vertices labeled $4, 5, 6, 7$ in $(G, \lambda')$.  Again, their labels are circularly shifted in $(G, \lambda'')$ to $7, 4, 5, 6$.

It is valid to ask whether at the beginning we could have hastened the transformation process by simply swapping the labels of the vertex labeled $7$ with the vertex labeled $9$.  If we did, then $9$ will indeed be a pinnacle and $7$ will not be but $8$ will also be a pinnacle. That is, the pinnacle set of the relabeled graph becomes $\{4, 8, 9, 10\}$ and not $\{4, 9, 10\}$.  

\end{example}

\begin{remark}
Theorem~\ref{bigthm2} provides another reason why (1) $\Rightarrow$ (2) in Theorem~\ref{bigthm1}.  Let $P = \{p_1, p_2, \ldots, p_k \}$ be a pinnacle set of $G$  with $p_1 < p_2 < \cdots < p_k$. 
One can readily see that 
$p_k \leq n$, $p_{k-1} \leq n-1, \ldots, p_1 \leq n-k+1$.  
By Theorem~\ref{bigthm2}, the set $\{n-k+1, \ldots, n-1, n\}$ has to be another pinnacle set of $G$.
\end{remark}

Lemma~\ref{nextpinnacleset} was about how to swap an element in a pinnacle set of $G$ with a larger one.   In Lemma~\ref{downpinnacleset} below, we consider the reverse of the process, swapping an element with a smaller one. For example, if $\{3, 6, 8, 10\}$ is a pinnacle set of $G$, and the vertices whose labels are 2 and 3 are not adjacent, then $\{2, 6, 8, 10\}$ is another pinnacle set of $G$.  
 The proof is quite similar to the one for Lemma~\ref{nextpinnacleset}.

\begin{lemma}\label{downpinnacleset}
Let $P = \{p_1, p_2, \ldots, p_k \}$ be a pinnacle set of $(G,\lambda)$ with $p_0 \leq p_1 < p_2 < \cdots < p_k$ and $p_0 = 2$.  Suppose that $p_i - p_{i-1} > 1$ for some $i$, $1 \leq i \leq k$.  Let $\lambda(w) = p_i$ and $\lambda(x) = p_i - 1$.  If $w$ and $x$ are not adjacent, then $(P \setminus \{p_i\}) \cup \{p_i - 1\}$ is a pinnacle set of $G$. 
\end{lemma}

\begin{proof}
Since $p_i \in P$, it follows that all the neighbors of $w$ have labels smaller than $p_i$ and are not in $P$. 
In fact, their labels are smaller than $p_i -1$ because $w$ and $x$ are not adjacent.   
On the other hand, since $p_{i} - 1 \not \in P$, $x$ must have a neighbor whose label is larger than $p_i -1$.  The said label has to be larger $p_i$ because $w$ and $x$ are not adjacent. Vertex $x$ may also have neighbors whose labels are smaller than $p_i - 1$.

Define $\lambda''$ as the label obtained from $\lambda$ by swapping the labels of $x$ and $w$ and let $P''$ be the pinnacle set of $(G, \lambda'')$. 
We note that when $v \neq w, x$ nor a neighbor of these two vertices,  the status of $\lambda(v) = \lambda''(v)$ in $P$ and $P''$ are the same.  

Consider $w$.  Its neighbors' labels did not change in $\lambda''$ so they are still smaller that $p_i - 1$.  But $\lambda''(w) = p_i -1$ so $p_i -1 \in P''$ and none of $w$'s neighbors' labels are in $P''$.  Now $x$ has a neighbor whose label is at least $p_i +1$ and $\lambda''(x) = p_i$ so $p_i \not \in P''$.  For the other neighbors $z$ of $x$, we now argue their status is the same in $\lambda$ and $\lambda'$.  If $\lambda(z) \in P$, then $\lambda''(z) = \lambda(z) > p_i$ and will not be affected by the change in label of $x$ (and $w$) so it will be in $P''$. If $\lambda(z) \not \in P$, then there are two possibilities: (i) $\lambda(z) < \lambda(x)$ or (ii) $\lambda(z) > \lambda(x)$ and $z$ has a neighbor whose label is larger than $\lambda(z)$.  For case (i),  since $\lambda''(z) = \lambda(z)$ but $\lambda''(x) > \lambda(x)$, it follows that $\lambda''(z) \not \in P''$.  For case (ii), $\lambda(z)$ and its neighbor had $\lambda$-labels larger than $p_i$.  Their labels did not change in $\lambda''$.  It follows that $\lambda''(z) = \lambda(z) \not \in P''$.  Thus, the pinnacle set of $(G, \lambda'')$ is $P'' = P - \{p_i\} \cup \{p_i - 1\}$. 
 \end{proof}

Unfortunately, we cannot prove a counterpart to Theorem~\ref{bigthm2} because Lemma~\ref{downpinnacleset} is dependent not just on the elements of $P$ but also on the labeling whose pinnacle set is $P$.

\subsection{Using ordered tree partitions of a graph}

We begin by introducing the notion of ordered tree partitions of a graph and some of its properties.  We will use it in Theorem~\ref{bigthm3} to relabel vertices so that we can derive a new pinnacle set from an old one.

Let $I = \{v_1, v_2, \ldots, v_k\}$ be the independent set of $G$ so that the labels $\lambda(v_i)$, $i = 1,2, \ldots, k$ are exactly the pinnacles of $(G,\lambda)$. 
Intuitively, this means we can partition $G$ into $k$ parts, where the $i$th part is the ``region of influence"  of $v_i \in I$.  That is, if $u$ is in this $i$th part, then there is a path from $u$ to $v_i$ so that the labels of the vertices along the path are increasing.   Thus, $v_i$'s region of influence is a connected subgraph of $G$ and $\lambda(v_i)$ is the only pinnacle in that region.  We shall describe this region by a tree $T_i$ rooted at $v_i$ and call the sequence of trees $(T_1, T_2, \ldots, T_k)$  with some extra conditions an {\it ordered tree partition of $G$}.  We define the structure of an ordered tree partition of $G$ without a labeling first.

The reader should be aware that when we discuss the relationship between two vertices $u$ and $v$, there are two contexts:  their relationship in $G$ and their relationship in the ordered tree partition of $G$.  When the context is $G$, we will say $u$ is {\it adjacent} to $v$ or $u$ is {\it neighbors} with $v$, etc.  On the other hand, when the context is a rooted tree in the ordered tree partition, we will say $u$ is a {\it child} of $v$ or $u$ is a {\it parent} of $v$.

\begin{definition}\label{def:ordered tree partition}
Let $G$ be a graph.  An {\it ordered tree partition} of $G$ is a sequence $\mathcal{T} = (T_1, T_2, \ldots, T_k)$ of subgraphs of $G$ where each $T_i$ is a tree rooted at $r_i$   so that 
\begin{itemize}
\item[(i)] $\{r_1, r_2, \ldots, r_k\}$ is an independent set of $G$

\item[(ii)] $(V(T_1), V(T_2), \ldots, V(T_k))$ is a partition of $V(G)$ and

\item[(iii)] when $u$ is neighbors in $G$ with at least one vertex in  $\{r_1, r_2, \ldots, r_k\}$, then $u$ is a child of $r_h$ with $h= \min\{i : \mbox{$u$ is adjacent to $r_i$ in $G$} \}$. 
\end{itemize}
\end{definition}

See Figures~\ref{otp1fig} and~\ref{otp2fig} for some examples. We emphasize that the ordering of the trees in $\mathcal{T}$ is important because of condition (iii).  The next proposition states some of its implications.

\begin{proposition}
Let $\mathcal{T} = (T_1, T_2, \ldots, T_k)$ be an ordered tree partition of $G$, and let $v \in V(T_i)$ with $v \neq r_i$.  
When $v$ is a child of $r_i$, then $v$ is not adjacent to $r_1, r_2, \ldots, r_{i-1}$ in $G$.  
On the other hand, when  $v$ is not a child of $r_i$, then $v$ is not adjacent to $r_1, r_2, \ldots, r_k$ in $G$. 
\end{proposition}

In the next two lemmas, we formalize the ideas described earlier  by stating explicit connections between a labeled graph's pinnacles and their corresponding vertices' regions of influence.

\begin{figure}
\begin{center}
 
     \begin{tikzpicture}[roundnode/.style={circle, draw=black,  minimum size=.5mm}]
    \node[roundnode,fill=green!20](A1) at (0,-2){3};
    \node at (-0.66,-2){$r_1$};
    \node[roundnode](A2) at (0,-1){1};
    \node[roundnode](A3) at (1,-1.5){2};
    \node[roundnode,fill=green!20](A4) at (2,-1.5){6};
    \node at (2,-2.2){$r_2$};
    \node[roundnode](A5) at (3,-1.5){4};
    \node[roundnode](A6) at (4,-1.5){5};
    \node[roundnode](A7) at (5,-2){7};
    \node[roundnode,fill=green!20](A8) at (5,-1){8};
    \node at (5.66,-1){$r_3$};
   
    \draw (A3)--(A2)--(A1)--(A3)--(A4)--(A5)--(A6)--(A7)--(A8)--(A6);
    \node[roundnode,fill=green!20](B1) at (0,-4.5){3};
    \node at (-0.66,-4.5){$r_1$};
    \node[roundnode](B2) at (0,-3.5){1};
    \node[roundnode](B3) at (1,-4){2};
    \node[roundnode,fill=green!20](B4) at (2,-4){6};
    \node at (2,-4.7){$r_2$};
    \node[roundnode](B5) at (3,-4){4};
    \node[roundnode](B6) at (4,-4){5};
    \node[roundnode](B7) at (5,-4.5){7};
    \node[roundnode,fill=green!20](B8) at (5,-3.5){8};
    \node at (5.66,-3.6){$r_3$};
    
    \draw[->] (B3)--(B1);
    \draw[->] (B2)--(B1);
    \draw (B2)--(B3)--(B4);
    \draw[->] (B5)--(B4);
    \draw (B5)--(B6)--(B7);
    \draw[->] (B6)--(B8);
    \draw[->] (B7)--(B8);

    \node[roundnode,fill=green!20](A1) at (8.5,-2){};
    \node at (9,-2){$r_1$};
    \node[roundnode,fill=green!20](A2) at (10.5,-2){};
    \node at (11,-2){$r_2$};
    \node[roundnode,fill=green!20](A3) at (12.5,-2){};
    \node at (13,-2){$r_3$};

    \node[roundnode](A12) at (8,-3){};
    \node[roundnode](A13) at (9,-3){};
    \node at (8.5,-4){$T_1$};

    \node[roundnode](A21) at (10.5,-3){};
    \node at (10.5,-4){$T_2$};

    \node[roundnode](A31) at (12,-3){};
    \node[roundnode](A32) at (13,-3){};
    \node at (12.5,-4){$T_3$};

    \draw [<-](A1)--(A12);
    \draw [<-](A1)--(A13);
    \draw [<-](A2)--(A21);
    \draw [<-](A3)--(A31);
    \draw [<-](A3)--(A32);

\end{tikzpicture}

\end{center}
\caption{Consider the labeling $\lambda$ of graph $G$ on the top left.  Its pinnacle set is $\{3, 6, 8\}$.  Denote as $r_1, r_2, r_3$ the vertices whose labels are $3, 6, 8$, respectively. In the bottom left, the vertices whose labels are not pinnacles are assigned {\it parents} according to the rules described in the proof of Lemma~\ref{otp1}. The result is an ordered tree partition $(T_1, T_2, T_3)$ of $G$ shown on the right.  According to Lemma~\ref{otp1}, $T_1$, $T_1 \cup T_2$ and $T_1 \cup T_2 \cup T_3$ are guaranteed to have at most $3, 6$, and $8$ vertices respectively.} 
\label{otp1fig}
\end{figure}
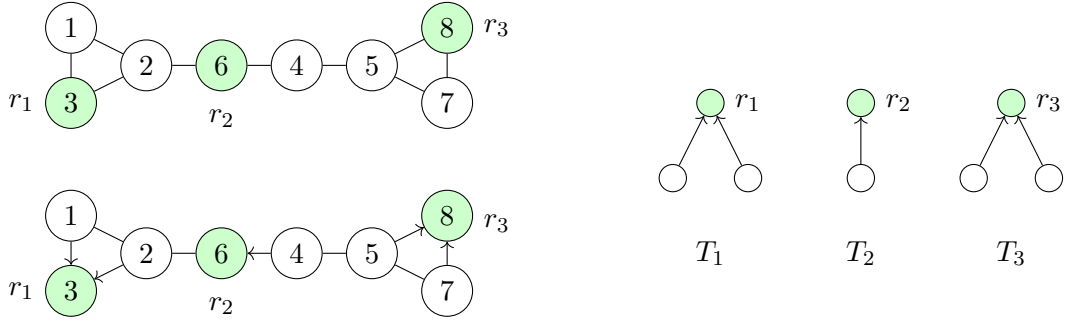

\begin{lemma}
\label{otp1}
Let  $\lambda$ be a labeling of $G$ so that  $Q = \{q_1, q_2, \ldots, q_k \}$ is a pinnacle set of $(G, \lambda)$ with $q_1 <  q_2 < \cdots < q_k$. Then $G$ has an ordered tree partition $\mathcal{T}_Q = (T_1, T_2, \ldots, T_k)$ so that   $ \sum_{j = 1}^i |V(T_j)| \leq q_i$ for all $1\leq i\leq k$.  \end{lemma}

\begin{proof}
For $i = 1, 2,\ldots, k$, let $r_i$ be the vertex whose label is $q_i$; i.e., $\lambda(r_i) = q_i$.  We now create an ordered tree partition of $G$ by assigning each $v \not \in \{r_1, r_2, \ldots, r_k\}$ a {\it parent} using the following rules:
\begin{enumerate}
\item If $v$ is not adjacent to any of the $r_i$'s in $G$, let $parent(v)$ be a neighbor of $v$ whose label is larger than $\lambda(v)$.  Such a neighbor exists because $\lambda(v)$ is not a pinnacle of $(G, \lambda)$.

\item Otherwise, let $h = \min\{i: \mbox{$v$ is adjacent to $r_i$ in $G$} \}$ and set $parent(v) = r_h$.  
\end{enumerate}
This $parent$ assignment  induces a forest on $G$ with exactly $k$ trees because $r_1, r_2, \ldots, r_k$ are the only vertices with  no parents.  For $i = 1, 2,\ldots, k$, let $T_i$ denote the tree whose root is $r_i$.   

Since $\{r_1, r_2, \ldots, r_k\}$ contains  all the vertices whose labels are pinnacles of $(G,\lambda)$, it follows that the set is an independent set of $G$.  Moreover, because every vertex not in  $\{r_1, r_2, \ldots, r_k\}$ has a parent, each one is part of some tree and the trees do not share any vertices.  Thus, $(V(T_1),  V(T_2),$ $ \ldots,  V(T_k))$ is a partition of $V(G)$.  Lastly, rule 2 above matches condition (iii) of the definition of ordered tree partitions.  Thus,  $(T_1, T_2, \ldots, T_k)$ is an ordered tree partition of $G$.  See Figure~\ref{otp1fig} for an example.

 Next, consider an arbitrary tree $T_i$ and $v \in V(T_i)$.  Rule 1 above guarantees that when $parent(v)$ exists, $\lambda(v) < \lambda(parent(v))$.  By transitivity, every vertex in $T_i$ has a label that is less than or equal to $\lambda(r_i) = q_i$. This same observation holds for trees $T_1, \ldots, T_{i-1}$.  In each of these trees, the labels of the vertices are at most the label of the root node.  But $q_1 < q_2 < \cdots < q_i$ too.  So it follows that all of the vertices in $T_1 \cup T_2 \cup \cdots \cup T_i$ have labels less than or equal to $q_i$.  Since these labels are distinct, we conclude that $\sum_{j = 1}^i |V(T_j)| \leq q_i$ for $i = 1,2, \ldots, k$. \end{proof}

We note that construction described in the proof of Lemma \ref{otp1} does not yield a unique ordered tree partition because the {\it parent} assignment of $v$, when $v$ is not adjacent to some $r_i$, is not unique.  Nonetheless, {\it any} ordered tree partition obtained by the construction satisfies the lemma. 

\begin{lemma}
\label{otp2} 
 Let $\mathcal{T} = (T_1, T_2, \ldots, T_k)$ be an ordered tree partition of $G$.  Let $p_i = \sum_{j = 1}^i |V(T_j)|$ for 
 $i = 1,2, \ldots, k$.   
 Then $G$ has a labeling $\lambda$ so that  $P = \{p_1, p_2, \ldots, p_k \}$ is the pinnacle set of $(G, \lambda)$. 
 Moreover, $\lambda$ has the following properties:   $\lambda(r_i) = p_i$ for $i = 1, 2,\ldots, k$ and  the vertices in $T_i$ are assigned the labels $p_{i-1} +1,  p_{i-1} + 2, \ldots, p_i$, where $p_0 = 0$, for $i = 1, 2,\ldots, k$. 
 \end{lemma}

\begin{proof}
Let $L_1 = \{1, 2, \ldots, p_1\}$ and for $i= 2$ to $k$, let $L_i = \{p_{i-1} +1,  p_{i-1} + 2, \ldots, p_i \}$.  Notice that the number of labels in $L_i$ is exactly the number of vertices in $T_i$.  We now construct a labeling $\lambda$ of $G$ sequentially:  for $i = 1$ to $k$,  do a basic labeling of $(T_i, L_i, \{r_i\})$ as described in Definition~\ref{def:basiclabeling}.  Thus, $\lambda(r_i) = p_i$ and the rest of the vertices in $T_i$ are assigned the labels $p_{i-1} +1,  p_{i-1} + 2$ to $p_{i} - 1$.  See Figure~\ref{otp2fig} for an example.

Next, consider the labels of the vertices in the context of graph $G$.  When $v \in V(T_i)$ and $v \neq r_i$, $\lambda(v)$ is definitely not a pinnacle of $(G,\lambda)$ because the parent of $v$ in $T_i$ has a larger label.  
When $v = r_i$,  let $w$ be one of its neighbors.  By condition (iii) of the definition of an ordered tree partition, $w$ is a child of $r_1, r_2, \ldots, r_{i-1}$ or $r_i$.   Let us say $w$ is a child of $r_h$.  Then $\lambda(w) <  \lambda(r_h) \leq  \lambda(r_i)$. Thus, $\lambda(r_i)$  is larger than the labels of its neighbors in $G$ so $\lambda(r_i)$ is a pinnacle of $(G, \lambda)$.  It follows that $\{\lambda(r_1), \lambda(r_2), \ldots, \lambda(r_k) \} = \{p_1, p_2, \ldots, p_k \}$ is exactly the pinnacle set of $(G,\lambda)$.
\end{proof}

The next theorem will take advantage of the labeling $\lambda$ in Lemma \ref{otp2} because it has a nice structure: whenever two vertices $u$ and $v$ belong to two different trees $T_i$ and $T_j$, $\lambda(u) < \lambda(v)$ if and only if $i < j$.

\begin{figure}[t]
\begin{center}
 \resizebox{5in}{!}{\begin{tabular}{p{2in}p{2in}}
    \begin{tikzpicture}[roundnode/.style={circle, draw=black, inner sep=0pt,  minimum size=5mm}]
        \node[roundnode](A1) at (0,0){$v_1$};
        \node[roundnode](A2) at (0,1){$v_2$};
        \node[roundnode](A3) at (1,0.5){$v_3$};
        \node[roundnode,fill=green!20](A4) at (2,0.5){$r_1$};
        \node[roundnode](A5) at (3,0.5){$v_4$};
        \node[roundnode](A6) at (4,0.5){$v_5$};
        \node[roundnode,fill=green!20](A7) at (5,1){$r_2$};
        \node[roundnode](A8) at (5,0){$v_6$};
        \draw (A1)--(A2);
        \draw (A5)--(A6)--(A8);
        \draw (A1)--(A3);
        \draw (A2)--(A3);
        \draw (A3)--(A4);
        \draw (A5)--(A4);
        \draw (A6)--(A7);
        \draw (A8)--(A7);
    \end{tikzpicture}
    &
    \qquad\qquad
    \begin{tikzpicture}[roundnode/.style={circle, draw=black, inner sep=0pt,  minimum size=5mm}]
        \node[roundnode,fill=green!20](A1) at (0,2){$r_1$};
        \node[roundnode](A2) at (-0.5,1){$v_3$};
        \node[roundnode](A3) at (0.5,1){$v_4$};
        \node[roundnode](A4) at (-1,0){$v_1$};
        \node[roundnode](A5) at (0,0){$v_2$};
        \draw[->] (A5)--(A2);
        \draw[->] (A4)--(A2);
        \draw[->] (A3)--(A1);
        \draw[->] (A2)--(A1);

        \node[roundnode,fill=green!20](B1) at (1.75,1){$r_2$};
        \node[roundnode](B2) at (1.25,0){$v_5$};
        \node[roundnode](B3) at (2.25,0){$v_6$};
        
        \draw[->] (B2)--(B1);
        \draw[->] (B3)--(B1);
    \end{tikzpicture}
    \\
&\\
    \begin{tikzpicture}[roundnode/.style={circle, draw=black, inner sep=0pt,  minimum size=5mm}]
        \node[roundnode,fill=green!20](A1) at (0,1){$5$};
        
        \node[roundnode](A2) at (-0.5,0){4};
        \node[roundnode](A3) at (0.5,0){3};
        \node[roundnode](A4) at (-1,-1){2};
        \node[roundnode](A5) at (-0,-1){1};
        \draw[->] (A5)--(A2);
        \draw[->] (A4)--(A2);
        \draw[->] (A3)--(A1);
        \draw[->] (A2)--(A1);

        \node[roundnode,fill=green!20](B1) at (1.75,0){$8$};
        \node[roundnode](B2) at (1.25,-1){7};
        \node[roundnode](B3) at (2.25,-1){6};
        \draw[->] (B2)--(B1);
        \draw[->] (B3)--(B1);
     \end{tikzpicture}&

     \begin{tikzpicture}[roundnode/.style={circle, draw=black, inner sep=0pt,  minimum size=5mm}]
     \node at (6.3,2.7){$T_1$};
     \node at (8.3,2.7){$T_2$};
         \node[roundnode](A1) at (4,0){2};
        \node[roundnode](A2) at (4,1){1};
        \node[roundnode](A3) at (5,0.5){4};
        \node[roundnode,fill=green!20](A4) at (6,0.5){5};
        \node[roundnode](A5) at (7,0.5){3};
        \node[roundnode](A6) at (8,0.5){7};
        \node[roundnode,fill=green!20](A7) at (9,1){8};
        \node[roundnode](A8) at (9,0){6};
        \draw (A1)--(A2);
        \draw (A5)--(A6)--(A8);
        \draw (A1)--(A3);
        \draw (A2)--(A3);
        \draw (A3)--(A4);
        \draw (A5)--(A4);
        \draw (A6)--(A7);
        \draw (A8)--(A7);
    \end{tikzpicture}
    \\
\end{tabular}}
\end{center}
\caption{On the top row is graph $G$ and one of its ordered tree partition $(T_1, T_2)$. 
On the second row, the vertices of the trees are labeled as described in the proof of Lemma~\ref{otp2}, resulting in the labeling of graph $G$ on the right. Since $T_1$ and $T_2$ have $5$ and $3$ vertices respectively, according to Lemma~\ref{otp2}, the pinnacles of the labeled graph are $5$ and $5+3 = 8$.}  
\label{otp2fig}
\end{figure}
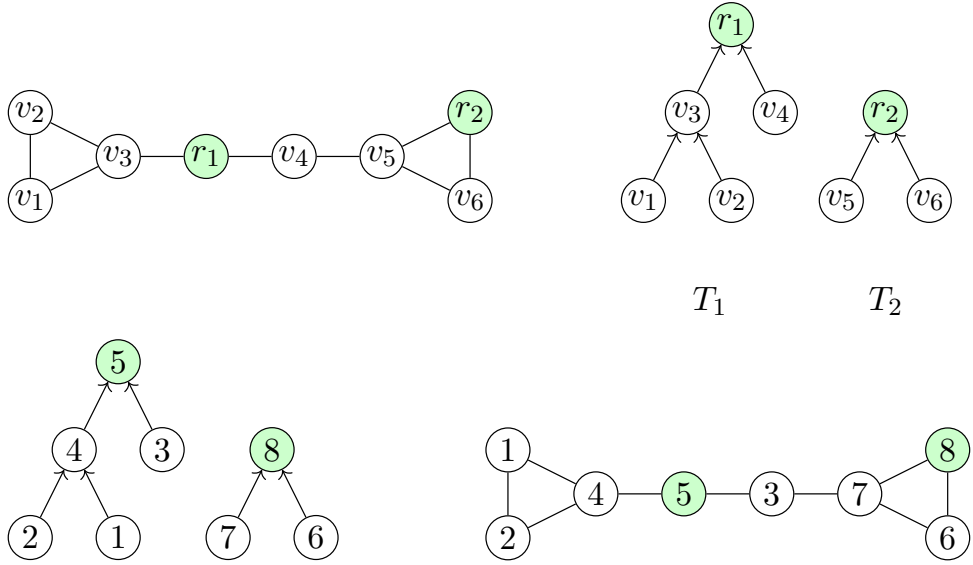

\begin{theorem}
\label{bigthm3}
Let  $\{q_1, q_2, \ldots, q_k \}$ be a pinnacle set of $G$ with $q_1 <  q_2 < \cdots < q_k$.  If $G$ is connected, then $\{q_i, q_{i+1}, \ldots, q_k\}$ is also a pinnacle set of $G$ for $i = 2,3, \ldots k$. 
\end{theorem}

For example, if $\{3, 6, 8, 10\}$ is a pinnacle set of $G$, then so are $\{6, 8, 10\}, \{8, 10\}$ and $\{10\}$. 

\begin{proof}
Let $\{q_1, q_2, \ldots, q_k \}$ be a pinnacle set of $G$. We will argue that $\{q_2, q_3, \ldots, q_k \}$ is also a pinnacle set of $G$. Applying the argument repeatedly, we conclude that $\{q_i, \ldots, q_k\}$ for $i = 2,3, \hdots, k$ are also pinnacle sets of $G$. 

By  Lemma~\ref{otp1}, $G$ has an ordered tree partition  $\mathcal{T} = (T_1, T_2, \ldots, T_k)$  so that for $p_i \leq q_i$ for $i = 1, 2,\ldots, k$, where $p_i = \sum_{j = 1}^i |V(T_j)|$.  
By Lemma~\ref{otp2}, $G$ has a labeling $\lambda$ so  that $\{p_1,  p_2, \ldots, p_k\}$ is the pinnacle set of $(G, \lambda)$.  Additionally, $\lambda$ has the property that (i) $\lambda(r_i) = p_i$ where $r_i$ is the root of $T_i$ for $i = 1, 2, \ldots, k$ and (ii) for any $u \in V(T_1)$ and $v \in V(T_j)$,  $\lambda(u) < \lambda(v)$ whenever $j > 1$.   In what follows, we will relabel the vertices of $T_1$ so that none of the labels in $T_1$, including $p_1$, are pinnacles.

Since $G$ is connected, there is an $x \in V(T_1)$ that is adjacent to some $y \in V(T_j)$, $j > 1$, in $G$.  Let $L_1$ be the set of labels used by $\lambda$ on the vertices of $T_1$. Recall that $L_1 = \{1, 2, \ldots, p_1\}$. 
We now relabel the vertices of $T_1$ as follows:   Root $T_1$ at $x$ instead of $r_1$.  Then apply the basic labeling of $(T_1, L_1, \{x\})$.  Keep the labeling of the vertices in $T_2, T_3, \ldots, T_k$ the same. Denote as $\lambda'$ this new labeling of $G$.  We will argue that the pinnacle set of $(G, \lambda')$ is $\{p_2, p_3, \ldots, p_k\}$.  

First, we note that because of how the basic labeling of $(T_1, L_1, \{x\})$ works, except for $\lambda'(x) =p_1$, no other labels in $L_1$ can be a pinnacle of $(G, \lambda')$.  But $x$ is adjacent to $y$ and $p_1 < \lambda(y) = \lambda'(y)$ because $y$ is in a tree other than $T_1$.   So it follows that $\lambda'(x) = p_1$ is {\it not} a pinnacle of $(G, \lambda')$. 

Next, consider $T_i$, $i > 1$.  Since the labels of its vertices were unchanged from $\lambda$ to $\lambda'$, it follows that for every vertex $v \in V(T_i)$ such that $v \neq r_i$, $\lambda'(v)$ is not a pinnacle of $(G, \lambda')$.  As for $r_i$, if it has a neighbor $w$ in $T_1$, we know that $\lambda'(r_i) = \lambda(r_i)$ is larger than all the labels in $L_1$; i.e., $\lambda'(r_i)> \lambda'(w)$.   Thus, $\lambda'(r_i) = p_i$ is still a pinnacle of $(G, \lambda')$.  

We have shown that $\{p_2, p_3, \ldots, p_k\}$ is a pinnacle set of $G$.  Since $p_i \leq q_i$ for $i = 2, 3,\ldots, k$, by Theorem~\ref{bigthm2}, $\{q_2, q_3, \ldots, q_k\}$ is a pinnacle set of $G$ too.  \end{proof}

In Figure \ref{otp2fig}, the graph labeling has the properties described in Lemma \ref{otp2} and its pinnacle set is $\{5,8\}$.  According to the proof of Theorem \ref{bigthm3}, to create a labeling whose pinnacle set is $\{8\}$, we need to a find a vertex $x$ in $T_1$ that is adjacent to a vertex $y$ in $T_2$. Thus, we let $x = v_4$ and $y = v_5$. $T_1$ is then re-oriented so that $v_4$ becomes the new root and a basic labeling starting at $v_4$ using the labels $1, 2, \hdots, 5$ is applied.  The result is shown in Figure \ref{otp3fig}.  We note that the pair formed by the modified $T_1$ and $T_2$ is no longer a valid ordered tree partition of the graph.

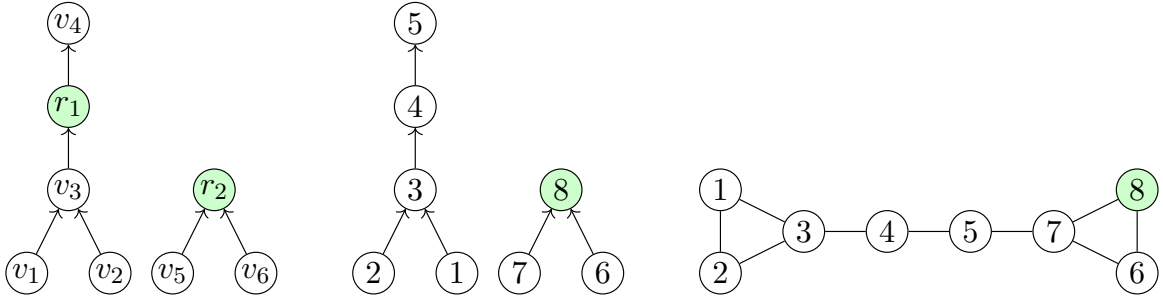
\begin{figure}[t]
\begin{center}
 \resizebox{6in}{!}{\begin{tabular}{p{1.5in}p{1.5in}p{2in}}
\begin{tikzpicture}[roundnode/.style={circle, draw=black, inner sep=0pt,  minimum size=5mm}]
        \node[roundnode,fill=green!20](A1) at (0,2){$r_1$};
        \node[roundnode](A2) at (0,1){$v_3$};
        \node[roundnode](A3) at (0,3){$v_4$};
        \node[roundnode](A4) at (-0.5,0){$v_1$};
        \node[roundnode](A5) at (0.5,0){$v_2$};
        \draw[->] (A5)--(A2);
        \draw[->] (A4)--(A2);
        \draw[->] (A1)--(A3);
        \draw[->] (A2)--(A1);

        \node[roundnode,fill=green!20](B1) at (1.75,1){$r_2$};
        \node[roundnode](B2) at (1.25,0){$v_5$};
        \node[roundnode](B3) at (2.25,0){$v_6$};
        \draw[->] (B2)--(B1);
        \draw[->] (B3)--(B1);
    \end{tikzpicture}
& 
\begin{tikzpicture}[roundnode/.style={circle, draw=black, inner sep=0pt,  minimum size=5mm}]
        \node[roundnode](A1) at (0,2){4};
        \node[roundnode](A2) at (0,1){3};
        \node[roundnode](A3) at (0,3){5};
        \node[roundnode](A4) at (-0.5,0){2};
        \node[roundnode](A5) at (0.5,0){1};
        \draw[->] (A5)--(A2);
        \draw[->] (A4)--(A2);
        \draw[->] (A1)--(A3);
        \draw[->] (A2)--(A1);

        \node[roundnode,fill=green!20](B1) at (1.75,1){8};
        \node[roundnode](B2) at (1.25,0){7};
        \node[roundnode](B3) at (2.25,0){6};
        \draw[->] (B2)--(B1);
        \draw[->] (B3)--(B1);
    \end{tikzpicture}
&
\begin{tikzpicture}[roundnode/.style={circle, draw=black, inner sep=0pt,  minimum size=5mm}]
         \node[roundnode](A1) at (4,0){2};
        \node[roundnode](A2) at (4,1){1};
        \node[roundnode](A3) at (5,0.5){3};
        \node[roundnode](A4) at (6,0.5){4};
        \node[roundnode](A5) at (7,0.5){5};
        \node[roundnode](A6) at (8,0.5){7};
        \node[roundnode,fill=green!20](A7) at (9,1){8};
        \node[roundnode](A8) at (9,0){6};
        \draw (A1)--(A2);
        \draw (A5)--(A6)--(A8);
        \draw (A1)--(A3);
        \draw (A2)--(A3);
        \draw (A3)--(A4);
        \draw (A5)--(A4);
        \draw (A6)--(A7);
        \draw (A8)--(A7);
    \end{tikzpicture}

\end{tabular}}
\end{center}
\caption{Continuing from Figure \ref{otp2fig}, $T_1$ is re-oriented so that $v_4$ is its new root.  The vertices of $T_1$ are then relabeled according to the proof of Theorem \ref{bigthm3} so that the new pinnacle set is $\{8\}$.}   
\label{otp3fig}
\end{figure}

\begin{remark} In the proof of Theorem \ref{bigthm3}, if  there is some vertex in $T_i$ that is adjacent to a vertex in $T_j$ with $1 \leq i < j \leq k$, then we can also conclude that $\{q_1, q_2, \ldots, q_{i-1}, q_{i+1}, \ldots q_k\}$.  The connectivity of $G$, however, does not guarantee this condition when $i \neq 1$. 
\end{remark}

    \section{The pinnacle posets of a graph}\label{sec:pinnacle poset}

Inspired by Theorem~\ref{bigthm2}, we  define a partial ordering (poset) on the pinnacle sets of a graph with the same size.   Let $\Pin(G,k)$ contain all the size-$k$ pinnacle sets of the graph $G$.  
For ease of discussion, we assume that the elements of a pinnacle set are ordered from smallest to largest. 
For any two sets $P, Q \in \Pin(G,k)$ with $P = \{p_1, p_2, \ldots, p_k\}$ and $Q = \{q_1, q_2, \ldots, q_k \}$, define the relation $\preceq$ as follows:  $P \preceq Q$ if and only if  $p_i \leq q_i$ for $i = 1, 2,\ldots, k$.  
It is easy to check that $\preceq$ is reflexive, anti-symmetric and transitive so $\mathcal{P}(G,k) = (\Pin(G,k), \preceq)$ is a partial order.  We note that Davis et al. (\cite{davis2018pinnacle}, Section 4) suggested a similar poset for the pinnacle sets of a permutation but did not explore it.

      Let $\mathcal{D} = (D, \leq)$ be a finite poset and  $d \in D$.  We say that $d$ is a {\it maximal element} of $\mathcal{D}$ if there is no $e \in D$ such that $d \leq e$ and $d \neq e$.  Additionally, if $d$ is the {\it only} maximal element of $\mathcal{D}$ or, equivalently,   for every $c \in D$, $c \leq d$,
 then $d$ is the {\it maximum element} of $\mathcal{D}$.   Similarly, $d$ is a {\it minimal element} of $\mathcal{D}$ if there is no $c \in D$ such that $c \leq d$ and $c\neq d$?  Futhermore, if $d$ is the {\it only} minimal element of $\mathcal{D}$ or, equivalently,  for every $e \in D$, $d \leq e$, then $d$ is the {\it minimum element} of $\mathcal{D}$.  Below, we show that when $\Pin(G,k) \neq \emptyset$, the poset $\Poset(G,k)$ has a maximum element. 
 
\begin{lemma}
Let $G$ be a graph with $n$ vertices.  Assume $\Pin(G,k) \neq \emptyset$.  Then $\Poset(G,k)$ has a maximum element: the set $M_{n,k} = \{n-k+1, n-k+2, \ldots, n\}$. 
\label{top1}
\end{lemma}

\begin{proof}
 When $\Pin(G,k) \neq \emptyset$, $G$ has a pinnacle set of size $k$. According to Theorem~\ref{bigthm1}, $M_{n,k}$ is one of them.  Furthermore, for any pinnacle set   $P = \{p_1, p_2, \ldots, p_k \} \in \Pin(G,k)$, $p_k \leq n, p_{k-1} \leq n-1, \ldots, p_1 \leq n-k+1$.  That is, $P \preceq M_{n,k}$.  Thus, $M_{n,k}$  is the maximum element of $\Poset(G,k)$. 
\end{proof}

      The poset $\mathcal{D} = (D, \leq)$ is a {\it join semilattice} if for any two elements $c, d \in D$, their {\it least upper bound} or {\it join} exists.    Similarly, if for $c, d \in D$,  their {\it greatest lower bound} or {\it meet}  exists, then $\mathcal{D}$ is a {\it meet semilattice}.  We denote the join and meet of $c$ and $d$ as $c \lor d$ and $c \land d$ respectively.

       When $\mathcal{D}$ is both a join and meet semilattice, then $\mathcal{D}$ is a {\it lattice}.  Finally,  $\mathcal{D}$ is a {\it distributive lattice} if for any three element $c, d, e \in D$, $c \land (d \lor e) = (c \land d) \lor (c \land e)$ or, equivalently, $c \lor (d \land e) = (c \lor d) \land (c \lor e)$.  For example,  for any set $A$, its power set $\mathcal{P}(A)$ forms a distributive lattice under the subset relation.  For $A_1, A_2 \in \mathcal{P}(A)$,  $A_1 \lor A_2 = A_1 \cup A_2$ while $A_1 \land A_2 = A_1 \cap A_2$.  The fact that the union and intersection operations distribute over each other is why $(\mathcal{P}(A), \subseteq)$ a distributive lattice and not just a lattice. We are now ready for our next result.

\begin{theorem}
\label{bigthm4}
Let $G$ be a graph with $n$ vertices. Assume $\Pin(G,k) \neq \emptyset$. Then $\Poset(G,k)$ is a join semilattice.  For any two pinnacle sets $P, Q \in \Pin(G,k)$ with $P = \{p_1, p_2, \ldots, p_k \}$ and $Q = \{q_1, q_2, \ldots, q_k \}$, their join is $P \lor Q = \{  \max\{p_i, q_i\}:  i = 1,2, \ldots, k \}.$
\end{theorem}

\begin{proof}
Let $P, Q \in \Pin(G,k)$ with $P = \{p_1, p_2, \ldots, p_k \}$ and $Q = \{q_1, q_2, \ldots, q_k \}$.  Define $\max(P, Q)$ as the set $\{  \max\{p_i, q_i\},  i = 1,2, \ldots, k \}$.  Let us argue that $\max(P,Q)$ is also a size-$k$ pinnacle set of~$G$.  
For $i = 1,2, \ldots, k-1$,  $p_i < p_{i+1}$ and $q_i < q_{i+1}$ by definition.  If $\max\{p_i, q_i\} = p_i$  then $\max\{p_i, q_i\} < p_{i+1} \leq \max\{p_{i+1}, q_{i+1}\}$.  Similarly, if $\max\{p_i, q_i\} = q_i$ then $\max\{p_i, q_i\} < q_{i+1} \leq \max\{p_{i+1}, q_{i+1}\}$.  Thus, $\max\{p_i, q_i\} < \max\{p_{i+1}, q_{i+1}\}$ which implies that the sequence $(\max\{p_1, q_1\}, \max\{p_2, q_2\}, \ldots,  \max\{p_k, q_k\})$ is strictly increasing. That is,  no two elements of $\max(P, Q)$ are equal so  $\max(P, Q)$ has size $k$.  Furthermore, {\it every} element of $\max(P, Q)$ is from $P \cup Q$ so $\max(P, Q)$ is a size-$k$ subset of $[n]$.  Now, 
for $i = 1, 2, \ldots, k$,  $p_i \leq \max\{p_i, q_i\}$.  From Theorem~\ref{bigthm2},  since $P$ is a size-$k$ pinnacle set of $G$,  so is $\max(P, Q)$.  

We have established that $\max(P, Q) \in \Pin(G,k)$.  Clearly, $P \preceq \max(P,Q)$ and $Q \preceq \max(P,Q)$ so $\max(P,Q)$ is an upper bound of $P$ and $Q$.  Now, suppose $S = \{s_1, s_2, \ldots, s_k\} \in \Pin(G,k)$ is another upper bound of $P$ and $Q$. 
 It follows that for  $i = 1,2, \ldots, k$, $\max\{p_i, q_i\} \leq s_i$.  Consequently, $\max(P,Q) \preceq S$   and  $P \lor Q = \max(P,Q)$.
 That is, $\Poset(G,k)$ is a join semilattice.  
\end{proof}

    Let us refer to the minimal elements of $\Poset(G,k)$ as the {\it bottom size-$k$ pinnacle sets of $G$}.   A natural question to ask is: when  $\Pin(G,k) \neq \emptyset$, is $\Poset(G,k)$ a meet semilattice too?  Figure~\ref{fig:two bottom pinnacles} implies the answer is no. It shows a graph that has two bottom size-3 pinnacle sets.  
    Nonetheless, there are exceptions. Figure~\ref{fig:poset for P9} displays the posets of size-$k$ pinnacles sets of the path $P_9$ for $k = 2$, $3$ and $4$.  All of them have minimum elements.  Corollaries \ref{cor:gaps for cycle} and \ref{cor:gaps for paths} imply that the posets of the size-$k$ pinnacle sets of  cycles and paths have minimum elements provided $\Pin(C_n,k)$ and $\Pin(P_n,k)$ are non-empty.

\begin{proposition}\label{prop:bottom pinnacle set of cycle}
Let $n \ge 3$ and $k \ge 2$ be positive integers such that $n \ge 2k$.  Then $\{3, 5, 7, \hdots, 2k-1, n\}$ is the minimum pinnacle set of $\mathcal{P}(C_n,k)$. 
\end{proposition}

\begin{proposition}\label{prop:bottom pinnacle set of path}
Let $n \ge 3$ and $k \ge 2$ be positive integers such that $n \ge 2k -1$.  Then  $\{2, 4, \hdots, 2k-2, n\}$ is the minimum pinnacle set of $\mathcal{P}(P_n,k)$. 
\end{proposition}

Indeed, when $\Poset(G,k)$ has a minimum element, we can say something even stronger about the poset's structure.

\begin{corollary}
\label{cor4}
Let $G$ be a graph with $n$ vertices. Assume $\Pin(G,k) \neq \emptyset$.  If $\Poset(G,k)$ has a minimum element,  then $\Poset(G,k)$ is a distributive lattice.  For any two pinnacle sets $P, Q \in \Pin(G,k)$ with $P = \{p_1, p_2, \ldots, p_k \}$ and $Q = \{q_1, q_2, \ldots, q_k \}$, their meet is $P \land Q = \{  \min\{p_i, q_i\}:  i = 1,2, \ldots, k \}.$
\end{corollary}

\begin{proof}
Again, let $P = \{p_1, p_2, \ldots, p_k \}$ and $Q = \{q_1, q_2, \ldots, q_k \}$ be size-$k$ pinnacle sets of $G$.  Let $\min(P,Q) = \{ \min\{p_i, q_i\},  i = 1,2, \ldots, k \}$. Similar to the proof of Theorem \ref{bigthm4}, it is straightforward to argue that $\min(P,Q)$ is a size-$k$ subset of $[n]$.  Assume $\Poset(G,k)$ has a minimum element, say $B = \{b_1, b_2, \ldots, b_k \}$.   By definition, $B \preceq P$ and $B \preceq Q$ so $b_i \leq \min\{p_i, q_i\}$ for $i =1, 2,\ldots, k$.  Since $B$ is a size-$k$ pinnacle set of $G$, by Theorem~\ref{bigthm2}, $\min(P,Q)$ is also a size-$k$ pinnacle set of $G$.  

By definition, $\min(P,Q) \preceq P$ and $\min(P,Q) \preceq Q$ so $\min(P,Q)$ is a lower bound of $P$ and $Q$.  If some $S = \{s_1, s_2, \ldots, s_k\} \in \Pin(G,k)$ has $S \preceq P$ and $S \preceq Q$, then $s_i \leq \min\{p_i, q_i\}$ for $i = 1, 2,\ldots, k$ so $S \preceq \min(P,Q)$.  Thus, $P \land Q = \min(P,Q)$ and $\Poset(G,k)$ is a meet semilattice.  Together with Theorem~\ref{bigthm4}, we conclude that $\Poset(G,k)$ is a lattice.

Finally, to prove that  $\Poset(G,k)$ is actually a distributive lattice, we need to show that for any three pinnacle sets $P, Q, S$ in $\Poset(G,k)$, $P \land (Q \lor S) = (P \land Q) \lor (P \land S)$.  It is equivalent to showing that $ \min(P, \max (Q,S)) = \max( \min(P, Q), \min(P, S))$, which in turn is equivalent to showing that for $i = 1, 2,\ldots, k$,  $\min \{p_i, \max \{q_i, s_i\} \} = \max \{ \min \{ p_i, q_i\}, \min \{p_i, s_i \} \}$.  Since $p_i, q_i$ and $s_i$ are integers and the $\max$ operation distributes over the $\min$ operation (and vice versa) for the set of real numbers,
the last equality has to be true. Therefore, $\Poset(G,k)$ is a distributive lattice.  \end{proof}

\begin{figure}
    \centering
    \begin{tikzpicture}[roundnode/.style={circle, draw=black, inner sep=0pt,  minimum size=2mm}]
    \draw[line width=0.1mm];
    \node[roundnode, fill=black] at (0,0) {};
    \node[roundnode, fill=black] at (1,0) {};
    \node[roundnode, fill=black] at (2,0) {};
    \node[roundnode, fill=black] at (4,0) {};
    \node[roundnode, fill=black] at (3,1) {};
    \node[roundnode, fill=black] at (3,-1) {};
    \draw (0,0)--(4,0)--(3,-1);
    \draw (3,1)--(4,0);
    \draw (2,0)--(3,1);
    \draw (3,-1)--(2,0);
\end{tikzpicture}
\hspace{0.5in}
\begin{tikzpicture}
\node(a0) at (0,1){$\{4,5,6\}$};
\node(a1) at (0,0){$\{3,5,6\}$};
\node(m1) at (-1,-1){$\{2,5,6\}$};
\node(m2) at (1,-1){$\{3,4,6\}$};
\draw(a0)--(a1)--(m1);
\draw(a1)--(m2);
\end{tikzpicture}
    \caption{An example of a graph with two bottom size-3 pinnacle sets.}
    \label{fig:two bottom pinnacles}
\end{figure}
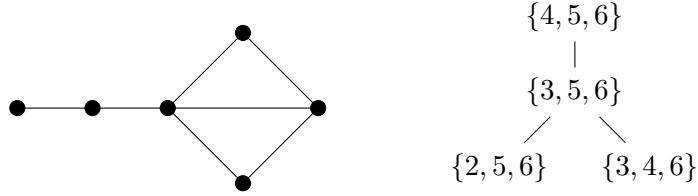

We end this subsection by strengthening Lemma~\ref{otp1} for bottom pinnacle sets. 

 \begin {lemma}
 \label{otp3}
Let $B = \{b_1, b_2, \ldots, b_k\}$ be a bottom size-$k$ pinnacle set of $G$.  Then $G$ has an ordered tree partition $\mathcal{T}_B = (T_1, T_2, \ldots, T_k)$ so that $b_i =  \sum_{j=1}^i |V(T_i)|$ for $i = 1, 2,\ldots, k$. 
\end{lemma}

\begin{proof}
Since $B$ is a pinnacle set of $G$, by Lemma~\ref{otp1},  $G$ has an ordered tree partition $(T_1, T_2, \ldots, T_k)$  so that $b_i \ge  \sum_{j=1}^i |V(T_i)|$ for $i = 1,2, \ldots, k$.  By Lemma~\ref{otp2}, $B' = \{b'_1, b'_2, \ldots, b'_k\}$, with $b_i' = \sum_{j=1}^i |V(T_i)|$ for $i = 1,2, \ldots, k$, is a pinnacle set of $G$.  Hence, $B' \preceq B$.  However, $B$ is a bottom pinnacle set of $G$ so $B' = B$. 
\end{proof}

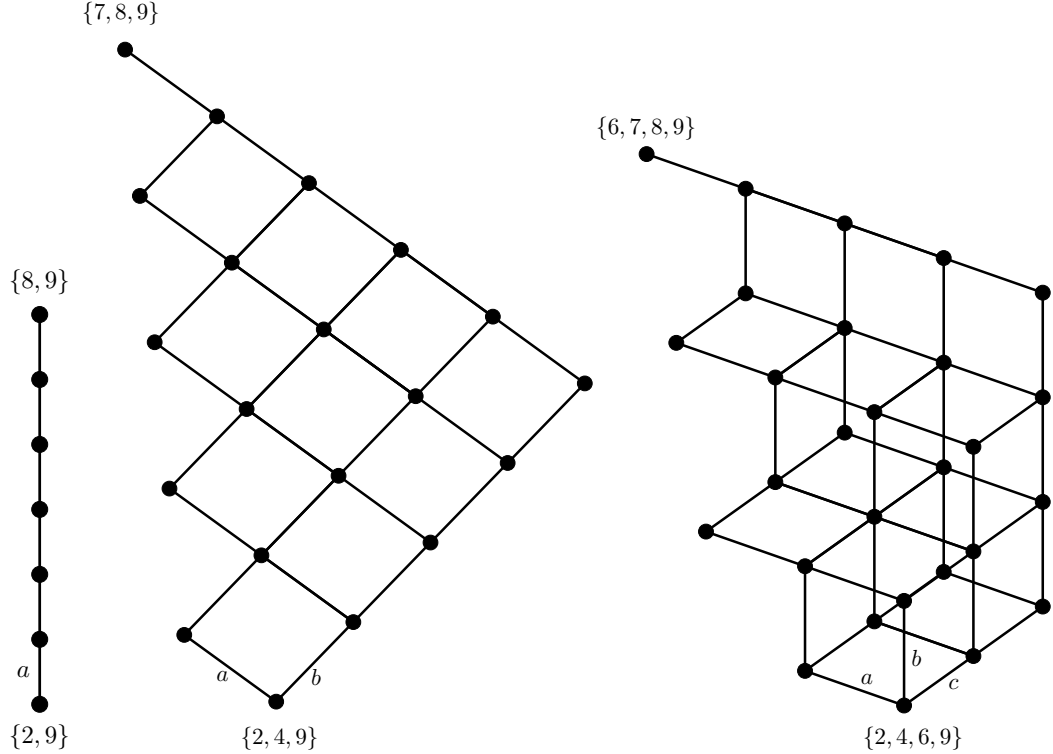
\begin{figure}
    \centering
    
    \begin{tabular}{cll}
    \usetikzlibrary {3d}
\tdplotsetmaincoords{60}{90}
\resizebox{!}{2.6in}{
\begin{tikzpicture}
		[tdplot_main_coords,
			cube/.style={very thick,black},
			grid/.style={very thin,gray},
			axis/.style={->,blue,thick}]
	 \draw[cube] (0,0,0) -- (12,0,0);
\node at(11,-.25,0){$a$};
\node at(13,0,0){$\{2,9\}$};
\node at(-1,0,0){$\{8,9\}$};
\draw[fill=black] (0,0,0) circle (3.5pt);
\draw[fill=black] (2,0,0) circle (3.5pt);
\draw[fill=black] (4,0,0) circle (3.5pt);
\draw[fill=black] (6,0,0) circle (3.5pt);
\draw[fill=black] (8,0,0) circle (3.5pt);
\draw[fill=black] (10,0,0) circle (3.5pt);
\draw[fill=black] (12,0,0) circle (3.5pt);
\end{tikzpicture}
}

&
\hspace{-3em}
\resizebox{!}{4in}{
\usetikzlibrary {3d}
\tdplotsetmaincoords{30}{130}
\begin{tikzpicture}
		[tdplot_main_coords,
			cube/.style={very thick,black},
			grid/.style={very thin,gray},
			axis/.style={->,blue,thick}]

	 \draw[cube] (0,0,0) -- (0,2,0) -- (2,2,0) -- (2,0,0) -- cycle;

\draw[cube] (2,0,0) -- (2,2,0) -- (4,2,0) -- (4,0,0) -- cycle;
\draw[cube] (0,-2,0) -- (0,0,0) -- (2,0,0) -- (2,-2,0) -- cycle;
\draw[cube] (-4,-4,0) -- (-4,-2,0) -- (-2,-2,0) -- (-2,-4,0) -- cycle;
\draw[cube] (-2,-4,0) -- (-2,-2,0) -- (0,-2,0) -- (0,-4,0) -- cycle;
\draw[cube] (-2,-2,0) -- (-2,0,0) -- (0,0,0) -- (0,-2,0) -- cycle;
\draw[cube] (-4,-2,0) -- (-4,0,0) -- (-2,0,0) -- (-2,-2,0) -- cycle;
\draw[cube] (-4,-6,0) -- (-4,-4,0) -- (-2,-4,0) -- (-2,-6,0) -- cycle;
	 \draw[cube](-4,-6,0)--(-4,-8,0);
  \draw[cube](-4,2,0)--(-0,2,0);
  \draw[cube](-4,2,0)--(-4,-2,0);
  \draw[cube](-2,2,0)--(-2,-2,0);
  
\node at(4.2,1,0){$a$};
\node at(3.2,2.2,0){$b$};
\node at(4.5,2.5,0){$\{2,4,9\}$};
\node at(-4.5,-8.5,0){$\{7,8,9\}$};
\draw[fill=black] (0,0,0) circle (3.5pt);
\draw[fill=black] (2,0,0) circle (3.5pt);
\draw[fill=black] (4,0,0) circle (3.5pt);
\draw[fill=black] (-2,0,0) circle (3.5pt);
\draw[fill=black] (-4,0,0) circle (3.5pt);
\draw[fill=black] (0,2,0) circle (3.5pt);
\draw[fill=black] (2,2,0) circle (3.5pt);
\draw[fill=black] (4,2,0) circle (3.5pt);
\draw[fill=black] (-2,2,0) circle (3.5pt);
\draw[fill=black] (-4,2,0) circle (3.5pt);

\draw[fill=black] (0,-2,0) circle (3.5pt);
\draw[fill=black] (2,-2,0) circle (3.5pt);
\draw[fill=black] (-2,-2,0) circle (3.5pt);
\draw[fill=black] (-4,-2,0) circle (3.5pt);
\draw[fill=black] (-4,-4,0) circle (3.5pt);
\draw[fill=black] (-2,-4,0) circle (3.5pt);
\draw[fill=black] (0,-4,0) circle (3.5pt);
\draw[fill=black] (-4,-6,0) circle (3.5pt);
\draw[fill=black] (-2,-6,0) circle (3.5pt);
\draw[fill=black] (-4,-8,0) circle (3.5pt);

\end{tikzpicture}

}
&
\hspace{-3em}
\resizebox{!}{3.4in}{
\usetikzlibrary {3d}

\tdplotsetmaincoords{60}{125}
\begin{tikzpicture}
		[tdplot_main_coords,
			cube/.style={very thick,black},
			grid/.style={very thin,gray},
			axis/.style={->,blue,thick}]



	 \draw[cube] (0,0,0) -- (0,2,0) -- (2,2,0) -- (2,0,0) -- cycle;
	
	\draw[cube] (0,0,0) -- (0,0,2);
	\draw[cube] (0,2,0) -- (0,2,2);
	\draw[cube] (2,0,0) -- (2,0,2);
	\draw[cube] (2,2,0) -- (2,2,2);

	\draw[cube] (0,0,2) -- (0,2,2) -- (2,2,2) -- (2,0,2) -- cycle;
	\draw[cube] (0,0,4) -- (0,2,4) -- (2,2,4) -- (2,0,4) -- cycle;
	
	\draw[cube] (0,0,2) -- (0,0,4);
	\draw[cube] (0,2,2) -- (0,2,4);
	\draw[cube] (2,0,2) -- (2,0,4);
	\draw[cube] (2,2,2) -- (2,2,4);

	 \draw[cube] (2,0,0) -- (2,2,0) -- (4,2,0) -- (4,0,0) -- cycle;

   \draw[cube] (2,-2,2) -- (2,0,2) -- (4,0,2) -- (4,-2,2) -- cycle;

 \draw[cube] (0,-4,4) -- (0,-2,4) -- (2,-2,4) -- (2,-4,4) -- cycle;

	  \draw[cube] (2,0,2) -- (2,2,2) -- (4,2,2) -- (4,0,2) -- cycle;
	
	
	 \draw[cube] (4,0,0) -- (4,0,2);
	 \draw[cube] (4,2,0) -- (4,2,2);

	 \draw[cube] (0,-2,2) -- (0,0,2) -- (2,0,2) -- (2,-2,2) -- cycle;
	\draw[cube] (0,-2,4) -- (0,0,4) -- (2,0,4) -- (2,-2,4) -- cycle;
	
	\draw[cube] (0,-2,2) -- (0,-2,4);
	
	  \draw[cube] (2,-2,2) -- (2,-2,4);
	  \draw[cube] (0,-2,4) -- (0,-2,6);
   \draw[cube] (0,0,4) -- (0,0,6);
   \draw[cube] (0,-4,4) -- (0,-4,6);
   \draw[cube](0,0,6)--(0,-4,6);
   \draw[cube](0,-6,6)--(0,2,6);
   \draw[cube](0,2,4)--(0,2,6);

\draw[fill=black] (0,0,0) circle (3.5pt);
\draw[fill=black] (2,0,0) circle (3.5pt);
\draw[fill=black] (4,0,0) circle (3.5pt);

\draw[fill=black] (0,2,0) circle (3.5pt);
\draw[fill=black] (2,2,0) circle (3.5pt);
\draw[fill=black] (4,2,0) circle (3.5pt);

\draw[fill=black] (0,2,2) circle (3.5pt);
\draw[fill=black] (2,2,2) circle (3.5pt);
\draw[fill=black] (4,2,2) circle (3.5pt);

\draw[fill=black] (0,0,2) circle (3.5pt);
\draw[fill=black] (2,0,2) circle (3.5pt);
\draw[fill=black] (4,0,2) circle (3.5pt);

\draw[fill=black] (0,-2,2) circle (3.5pt);
\draw[fill=black] (2,-2,2) circle (3.5pt);
\draw[fill=black] (4,-2,2) circle (3.5pt);

\draw[fill=black] (0,-4,4) circle (3.5pt);
\draw[fill=black] (2,-4,4) circle (3.5pt);
\draw[fill=black] (0,-2,4) circle (3.5pt);
\draw[fill=black] (2,-2,4) circle (3.5pt);
\draw[fill=black] (0,0,4) circle (3.5pt);
\draw[fill=black] (2,0,4) circle (3.5pt);
\draw[fill=black] (0,2,4) circle (3.5pt);
\draw[fill=black] (2,2,4) circle (3.5pt);

\draw[fill=black] (0,0,6) circle (3.5pt);
\draw[fill=black] (0,-2,6) circle (3.5pt);
\draw[fill=black] (0,-4,6) circle (3.5pt);

\draw[fill=black] (0,-6,6) circle (3.5pt);
\draw[fill=black] (0,2,6) circle (3.5pt);

\node at (4.8,2.75,0){$\{2,4,6,9\}$};
\node at (0,-6,6.5){$\{6,7,8,9\}$};

\node at(4,1.25,.25){$a$};
\node at(4,2.25,1){$b$};
\node at(4,3,.75){$c$};
\end{tikzpicture}

}
\end{tabular}
    \caption{The posets $\Poset(P_9,2)$, $\Poset(P_9,3)$, and $\Poset(P_9,4)$. All three are distributive lattices because they have a minimum element. In these figures, edges in the same direction as those labeled $a$ increase by one the first element in the set, edges going in the same direction as those labeled $b$ increase by one the second element of the set, and  edges in the same direction as those labeled $c$ increase by one the third element of the set.}
    \label{fig:poset for P9}
\end{figure}

 \subsection{Minimum pinnacle sets}

The structure of $\Poset(G,k)$ implies that the bottom size-$k$ pinnacle sets of $G$ play an important role.  Once determined, we know {\it all} of the size-$k$ pinnacle sets of $G$. That is, every size-$k$ subset $P$ of $[n]$ is a pinnacle set of $G$ if and only if there is some bottom size-$k$ pinnacle set $B$ of $G$ such that $B \preceq P \preceq M_{n,k}$.  The characterization is even more succinct when $G$ has a {\it unique} bottom size-$k$ pinnacle set, i.e., $\Poset(G,k)$ has a minimum element.  Below, we present two such cases.  For the first case, $G$ has $k$ connected components.  By Corollary~\ref{bigthmcor2}, the smallest-sized pinnacle set of $G$ is also $k$.

\begin{theorem}\label{thm:min element}
Suppose $G$ has $n$ vertices and $k$ connected components.  Let the sizes of these components be $h_1, h_2, \ldots, h_k$ respectively with $h_1 \leq h_2 \leq \cdots \leq h_k$.  Then $P = \{p_1, p_2, \ldots, p_k\}$ where $p_i = \sum_{j=1}^i h_j$ for $i = 1,2, \ldots, k$ is the minimum element of $\Poset(G, k)$.  
\label{bottomthm1}
\end{theorem}

\begin{proof}
     Let $H_1, H_2, \ldots, H_k$ be the connected components of $G$ where $|V(H_i)| = h_i$, $i = 1,2, \ldots, k$.  As stated in the theorem, let $h_1 \leq h_2 \leq \cdots \leq h_k$.  For each $H_i$, pick a vertex $r_i$ and run the breadth-first search (BFS) algorithm starting at $r_i$.  Let $T_i$ denote the resulting BFS-tree rooted at $r_i$.  By construction, all the neighbors of $r_i$ in $H_i$ will be its children in $T_i$.  It is straightforward to check that $(T_1, T_2, \ldots, T_k)$ is an ordered tree partition of $G$.  Consequently, by Lemma~\ref{otp2}, $P = \{p_1, p_2, \ldots, p_k\}$ is a pinnacle set of $G$.

     Let $Q = \{q_1, q_2, \ldots, q_k\}$ with $q_1 < q_2 < \cdots q_k$ be a pinnacle set of $G$.  Let $\lambda'$ be a labeling of $G$ so that $Q$ is the pinnacle set of $(G, \lambda')$.
     We will now argue that $P \preceq Q$.  Let $w_i$ be the vertex such that $\lambda'(w_i) = q_i$.   Let $F_i$ be the connected component that contains $w_i$, and let its size be $f_i$.  Since there are $k$ components and the pinnacle set of $(G, \lambda')$ has size $k$,  it has to be the case that $q_i$ is the largest label assigned by $\lambda'$ to the vertices in $F_i$ for $i = 1,2, \ldots, k$. In fact, we can say something stronger -- because $q_1 < q_2 < \cdots < q_i$, the largest label assigned by $\lambda'$ to all the vertices in $\bigcup_{j = 1}^i F_j$ is $q_i$.  Thus, $\sum_{j=1}^i f_j \leq q_i$.

     But $F_1, F_2, \ldots, F_k$ is a reordering of $H_1, H_2, \ldots, H_k$ and the sizes of the latter are ordered from smallest to largest.   It follows that $\sum_{j=1}^i h_j \leq \sum_{j=1}^i f_j$ because the $i$ smallest components of $G$ have sizes $h_1, h_2, \ldots, h_i$ respectively. Hence $p_i = \sum_{j=1}^i h_j \leq \sum_{j=1}^i f_j \leq q_i$  for $i = 1, 2,\ldots, k$.  Thus, $P \preceq Q$.  Since $Q$ is an arbitrary size-$k$ pinnacle set of $G$,  $P$ is the minimum element of $\Poset(G,k)$.  \end{proof}

\begin{corollary}
Suppose $G$ is a graph with $n$ vertices and $k$ connected components.  Let the sizes of these components be $h_1, h_2, \ldots, h_k$ with $h_1 \leq h_2 \leq \cdots \leq h_k$.  
Then $\Poset(G,k)$ is a distributive lattice whose maximum element is $M_{n,k} = \{n-k+1, \ldots, n-1, n\}$ and whose minimum element is $\{p_1, p_2, \ldots, p_k\}$ where $p_i = \sum_{j=1}^i h_j$ for $i = 1,2,\ldots, k$.   Equivalently, every size-$k$ subset $Q =  \{q_1, q_2, \ldots, q_k\}$ of $[n]$ is a pinnacle set of $G$ if and only if  $p_i \leq q_i \leq n-k + i$ for $i = 1,2,\ldots, k$.
\end{corollary}

     Next, assume $G$ has  size-$2$ pinnacle sets.  If $G$ has $n$ vertices, then every size-$2$ pinnacle set of $G$ has the form $\{s, n\}$ with $1 < s < n$. Thus, the poset $\Poset(G,2)$ is a linear order and therefore has a minimum element.  Now, there are only two possibilities for $G$: when $G$ is connected or $G$ has two connected components. Theorem~\ref{bottomthm1} identifies the minimum element of $\Poset(G,2)$ when $G$ has two connected components.  In the next theorem, we address first case when $G$ is connected.

 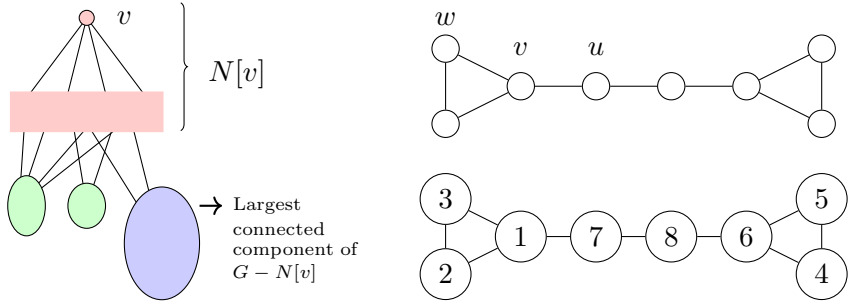
\begin{figure}
\begin{center}
 \begin{tikzpicture}
 [roundnode/.style={circle, draw=black, inner sep=0pt,  minimum size=2mm}]
    \draw[line width=0.1mm];
    \draw (0,0)--(-.75,-1);
    \draw (0,0)--(-.25,-1);
    \draw (0,0)--(.25,-1);
    \draw (0,0)--(.75,-1);
    \node[roundnode, fill=red!20] at (0,0) {};
    \node[] at (.5,0){$v$};
    \draw(-.9,-2.5)--(-.8,-1);
    \draw(-.9,-2.5)--
    (-.4,-1);
    \draw(-.9,-2.5)--
    (.4,-1);
    \draw(-.9,-2.5)--
    (1,-1);
    \draw(0,-2.5)--
    (.5,-1);
    \draw(0,-2.5)--
    (-.25,-1);
    \draw(1,-3)--
    (-.25,-1);
    \draw(1,-3)--
    (.5,-1);
    \draw[red!20, very thick, fill=red!20] (-1,-1) rectangle (1,-1.5);
    \draw[fill=green!20] (-.8,-2.5) ellipse 
    (.25cm and .4cm);
    \draw[fill=green!20] (0,-2.5) ellipse (.25cm and .3cm);
    \draw[fill=blue!20] (1,-3) ellipse (.5cm and .75cm);
    \draw[->, thick](1.5,-2.5)--(1.8,-2.5);
    \node at (2.4,-2.5){\tiny{Largest}};
    \node at (2.55,-2.8){\tiny{connected}};
    \node at (2.77,-3.1){\tiny{component of}};
    \node at (2.5,-3.4){\tiny{$G-N[v]$}};
    \node at (2,-.75){$N[v]$};
    \draw [decoration={brace}, decorate] (1.25,.2) -- node [pos=0.5,style={left=3pt}] {\small{}} (1.25,-1.5);
 \end{tikzpicture}
\hspace{0in}
\begin{tikzpicture}[roundnode/.style={circle, draw=black,  minimum size=.2mm}]
    \node[roundnode](A1) at (0,0){};
    \node at (-0.66,0){};
    \node[roundnode](A2) at (0,1){};
    \node at (0,1.4){$w$};
    \node[roundnode](A3) at (1,0.5){};
    \node at (1,1){$v$};
    \node[roundnode,](A4) at (2,0.5){};
    \node at (2,1){$u$};
    \node[roundnode](A5) at (3,0.5){};
    \node[roundnode](A6) at (4,0.5){};
    \node[roundnode](A7) at (5,1){};
    \node[roundnode](A8) at (5,0){};
    \node at (5.66,0){};
    \draw (A3)--(A2)--(A1)--(A3)--(A4)--(A5)--(A6)--(A7)--(A8)--(A6);
   \node[roundnode](B1) at (0,-1){3};
    \node at (-0.66,0){};
    \node[roundnode](B2) at (0,-2){2};
    \node[roundnode](B3) at (1,-1.5){1};
    \node[roundnode,](B4) at (2,-1.5){7};
    \node[roundnode](B5) at (3,-1.5){8};
    \node[roundnode](B6) at (4,-1.5){6};
    \node[roundnode](B7) at (5,-1){5};
    \node[roundnode](B8) at (5,-2){4};
    \node at (5.66,0){};
    \draw (B3)--(B2)--(B1)--(B3)--(B4)--(B5)--(B6)--(B7)--(B8)--(B6);
\end{tikzpicture}

\end{center}
\caption{The image on the left identifies the largest connected component of $G-N[v]$ for some vertex $v$.  Its size is $\alpha_v$.  For the graph on the top right, it is easy to check that $\alpha_u = 3$, $\alpha_v = 4$ while $\alpha_w = 5$.  Due to the symmetry of the graph, all other vertices $x$ of the graph will have $\alpha_x$ equal to $\alpha_u, \alpha_v$ or $\alpha_w$.  According to Theorem \ref{bottomthm2}, $s^* = 8 - 5 = 3$ so the graph's minimum size-$2$ pinnacle set is $\{3, 8\}$, which is achieved by the labeling  at the bottom right.}
\label{Size2fig}
\end{figure}

\begin{theorem}
Let $G$ be a connected graph with $n \ge 2$ vertices.  $G$ has a size-$2$ pinnacle set if and only if $G \neq K_n$.  Furthermore,   
when  $G \neq K_n$, the poset $\Poset(G,2)$ is a linear order whose  minimum element is $\{s^*, n\}$,  where  $s^* = n - \max\{\alpha_v, v \in V(G)\}$ and $\alpha_v$ is the size of the largest connected component of $G-N[v]$, the subgraph obtained by removing $v$ and its neighbors.
\label{bottomthm2}
\end{theorem}

\begin{proof}
Every connected graph with $n \ge 2$ vertices has an independent set of size two if and only if it has a pair of non-adjacent vertices.  The latter is true if and only if the graph is not a complete graph.  Let $G$ be one such graph.  

Since $G \neq K_n$, there is at least one vertex of $G$ that is not adjacent to all the vertices of $G$.  Let $v$ be one such vertex.  We will show that $\{n- \alpha_v, n\}$ is a pinnacle set of $G$.  By our choice of $v$, $G - N[v]$ is not empty.  Let $C$ be the largest connected component of $G - N[v]$, and let $w$ be a vertex in $C$.  Notice that $\{v, w\}$ is an independent set of $G$.  Furthermore, $G-C$ is a connected subgraph of $G$ since $G$ is connected so there must be edges between every connected component of $G-N[v]$ and $N[v]$.  See Figure~\ref{Size2fig} for an illustration of $G, N[v]$ and $C$.  

Run the BFS algorithm on $G - C$ starting at $v$. Let $T_v$ be the resulting BFS-tree rooted at $v$.  $T_v$ has  $n - |V(C)|$ vertices.  Next, run the breadth-first search algorithm on $C$ starting at $w$.  Let $T_w$ be the resulting BFS-tree rooted at $w$. $T_w$  has size $|V(C)|$.  It is easy to verify that $(T_v, T_w)$ satisfies conditions (i) and (ii) of the ordered tree partition of $G$.  Furthermore, condition (iii) is satisfied because we used the BFS algorithm so all the neighbors of $v$ are its children in $T_v$ and all the neighbors of $w$ not in $N[v]$ are its children in $T_w$.   

By Lemma~\ref{otp2}, since $(T_v, T_w)$ is an ordered tree partition of $G$,  $\{n- |V(C)|, n\} = \{n - \alpha_v, n\}$ is a pinnacle set of $G$.  Let $v^*$ be a vertex of $G$ such that $\alpha_{v^*} = \max \{\alpha_v, v \in V(G)\}$.  The same argument above implies that $\{n - \alpha_{v^*}, n\}$ is a pinnacle set of $G$.  Now if the minimum element of $\Poset(G,2)$ is $\{s^*, n\}$, we will argue next that $ s^* \ge n - \alpha_{v^*}$ so it must be the case that  $ s^* = n - \alpha_{v^*}$. 

By Lemma~\ref{otp3}, because $\{s^*, n\}$ is a bottom size-2 pinnacle set of $G$, there is an ordered tree partition $(T_1, T_2)$ of $G$ so that  $s^* = |V(T_1)| = n - |V(T_2)|$.  Let $r_1$ and $r_2$ be the roots of $T_1$ and $T_2$ respectively.  By definition, the neighbors of $r_1$ are its children in $T_1$.  So consider $G - N[r_1]$.  The graph cannot be empty because $r_2$ exists and $\{r_1, r_2\}$ form an independent set.   Assume $G - N[r_1]$ has one or more connected components and $r_2$ is in the connected component $D$.  Then $T_2$ has to be a subgraph of $D$ since $T_2$ is connected.  It follows that $|V(T_2)| \leq |V(D)|$.  But $D$ is one of the connected components of $G - N[r_1]$ so $|V(D)| \leq \alpha_{r_1}$.  Consequently, $|V(T_2)|  \leq \alpha_{r_1} \leq \alpha_{v^*}$ so $s^* = n - |V(T_2)| \ge n - \alpha_{v^*}$.   We have now established that  $\{n - \alpha_{v^*}, n\}$ is the minimum element of $\Poset(G,2)$. 
   \end{proof}

In Figure~\ref{Size2fig}, consider the graph $G$ on the top right. To compute $s^*$ we only have to consider the vertices $u, v$ and $w$ because of the symmetries of the graph.  It is easy to check that $\alpha_u = 3$, $\alpha_v = 4$ while $\alpha_w = 5$.  Hence, the minimum element of $\Poset(G,2)$ is $\{8-5, 8 \} = \{3, 8\}$.

\subsection{Counting distinct $k$-size pinnacle sets.} Let $\pinn(G,k)$ be the number of size-$k$ pinnacle sets of $G$; i.e., $\pinn(G,k) = |\Pin(G,k)|$.  We show how to compute $\pinn(G,k)$ when $G$ has a unique bottom size-$k$ pinnacle set.

\begin{theorem}\label{thm:count pinnacle sets}
    If $B=\{b_1,b_2,\ldots,b_k\}$ is only the bottom pinnacle of size $k$ of the graph $G$, then 
    \[\pinn(G,k)=\sum_{i_{k-1}=b_{k-1}}^{b_k-1}\cdots \sum_{i_{2}=b_{2}}^{i_3-1}\sum_{i_1=b_1}^{i_2-1}1\]
    is the total number of distinct pinnacle sets of size $k$. 
\end{theorem}
\begin{proof}
    From Lemma~\ref{nextpinnacleset}, \(\{b_1,\dots,b_{j-1},i,b_{j+1},\dots,b_k\}\) is also a pinnacle set where \(i\in [b_j,b_{j+1}-1]\). Notice that every time \(b_j\) is incremented by 1 we obtain a new pinnacle set of size \(k\). In order to count the total number of distinct pinnacle sets of size \(k\), we let \(i_j=b_j\) for \(1\leq j<k\) and proceed to increment \(i_1\) until \(i_1=i_2-1\). We select \(i_t\) such that \(i_t\) can be incremented by at least 1 and \(t\) is as small as possible. When \(i_t\) is incremented by 1, we reset \(i_j=b_j\) for \(1\leq j<t\). We then repeat the process of incrementing the smallest \(i\) such that \(i\) can be incremented. This iterative process is repeated until \(i_{k-1}=b_k-1,i_{k-2}=i_{k-1}-1,\dots,i_1=i_2-1\). This yields the total number of distinct pinnacle sets of size \(k\) along with the following summation as desired: \[\sum_{i_{k-1}=b_{k-1}}^{b_k-1}\cdots \sum_{i_{2}=b_{2}}^{i_3-1}\sum_{i_1=b_1}^{i_2-1}1.\qedhere\] 
\end{proof}

As an application of Theorem \ref{thm:count pinnacle sets} we give formulas for $\pinn(C_n,k)$ and $\pinn(P_n,k)$, 
the number of pinnacle sets of size $k$ for the cycle and path graphs. 
Before stating those results we recall that for a nonnegative integer $a$ and negative integer $b$, $\binom{a}{b}=0$.

\begin{corollary}\label{cor:count pinnacles of cycle graph}
Let $n\geq 3$ and $1\leq k\leq \lfloor\frac{n}{2}\rfloor$.
    Then 
    $\pinn(C_n,k)=\binom{n-2}{k-1} - \binom{n-2}{k-2}$.    
\end{corollary}

\begin{proof}
If $n=3$, then $k=1$, and the bottom pinnacle set of size $1$ is given by $\{3\}$, hence $\pinn(C_3,1)=1$. Note that $\binom{3-1}{1-1}-\binom{3-2}{1-2}=1-0=1$, as desired. 
Assume that for all $\ell\leq n$ and for all $1\leq k\leq \lfloor\frac{n}{2}\rfloor$, 
we have $\pinn(C_n,k)=\binom{n-2}{k-1}-\binom{n-2}{k-2}$.
Now consider the number of pinnacle sets of size $k$ for the cycle on $n+1$ vertices.
By Proposition \ref{prop:bottom pinnacle set of cycle} the bottom pinnacle set of size $k$ for  $C_n$ is given by $\{3,5,\ldots,2k-1,n\}$, while the bottom pinnacle set of size $k$ for $C_{n+1}$ is $\{3,5,\ldots,2k-1,n+1\}$.
By Theorem~\ref{thm:count pinnacle sets} we have that 
\begin{align*}
    \pinn(C_{n+1},k)&=\sum_{i_{k-1}=2k-1}^{(n+1)-1}\sum_{i_{k-2}=2k-3}^{i_{k-1}-1}\cdots \sum_{i_{2}=5}^{i_3-1}\sum_{i_1=3}^{i_2-1}1\\
    &=\sum_{i_{k-1}=2k-1}^{n}\sum_{i_{k-2}=2k-3}^{i_{k-1}-1}\cdots \sum_{i_{2}=5}^{i_3-1}\sum_{i_1=3}^{i_2-1}1\\
    &=\underbrace{\sum_{i_{k-2}=2k-3}^{n-1}\cdots \sum_{i_{2}=5}^{i_3-1}\sum_{i_1=3}^{i_2-1}1}_{i_{k-1}=n}+\sum_{i_{k-1}=2k-1}^{n-1}\sum_{i_{k-2}=2k-3}^{i_{k-1}-1}\cdots \sum_{i_{2}=5}^{i_3-1}\sum_{i_1=3}^{i_2-1}1\\
    &=\pinn(C_n,k-1)+\pinn(C_n,k)&\mbox{(by Theorem \ref{thm:count pinnacle sets})}\\
    &=\left(\binom{n-2}{k-2}-\binom{n-2}{k-3}\right)+\left(\binom{n-2}{k-1}-\binom{n-2}{k-2}\right)&\mbox{(by induction hypothesis)}\\
    &=\left(\binom{n-2}{k-1}+\binom{n-2}{k-2}\right)-\left(\binom{n-2}{k-2}+\binom{n-2}{k-3}\right)&\mbox{(by rearranging the terms)}\\
    &=\binom{n-1}{k-1}-\binom{n-1}{k-2},&\mbox{(by Pascal's identity)}
\end{align*}
as claimed.
\end{proof}

\begin{corollary}\label{couting all pinnacle sets of Cn}
For $n\geq 3$, the number of distinct pinnacle sets of $C_n$ is given by 
    \[\pinn(C_n)=\binom{n-2}{\lfloor\frac n2\rfloor-1}.\]
\end{corollary}
\begin{proof}
    This follows directly from Corollary \ref{cor:count pinnacles of cycle graph} and taking a sum over $1\leq k\leq \lfloor\frac n2\rfloor$.
\end{proof}

\begin{corollary}\label{cor:count pinnacles of path graph}
Let $n\geq 3$ and $1\leq k\leq \lceil\frac{n}{2}\rceil$.
    The $\pinn(P_n,k)=\binom{n-1}{k-1} - \binom{n-1}{k-2}$.    
\end{corollary}
\begin{proof}
    The proof is analogous to that of Corollary \ref{cor:count pinnacles of cycle graph} by using Proposition \ref{prop:bottom pinnacle set of path}, which states that the bottom pinnacle of $P_n$ is $\{2,4,\ldots,2(k-1),n\}$.
\end{proof}

\begin{corollary}\label{couting all pinnacle sets of Pn}
For $n\geq 3$, the number of distinct pinnacle sets of $P_n$ is given by 
    \[\pinn(P_n)=\binom{n-1}{\lceil\frac n2\rceil-1}.\]
\end{corollary}
\begin{proof}
    This follows directly from Corollary \ref{cor:count pinnacles of path graph} and taking a sum over $1\leq k\leq \lceil\frac n2\rceil$.
\end{proof}

\section{Future work}\label{sec:future}

Given our introduction of the pinnacle sets of a graph along with their orderly structure, much remains unknown about these sets. In this section we present some directions for further study.

Having characterized when a graph has a pinnacle set of a given size, an initial natural question is: 
 \begin{question}
     Given a graph and a graph operation, how does the size of a pinnacle set change?
 \end{question}
 We provide a result when we delete an edge of a graph. 
 \begin{lemma}
     If $S$ is a pinnacle set of $G$, then for any edge $e$ of $G$, the pinnacle set $T$ of $G-\{e\}$ satisfies $|S|\leq  |T|\leq |S|+1$.
 \end{lemma}
 \begin{proof}
     Consider an arbitrary edge $e=(u,v)$ of $G$ with $u$ labeled $a$ and $v$ labeled $b$, and assume without loss of generality that $a<b$. 
     Then, either both $a$ and $b$ are non-pinnacles, or exactly one of them is a pinnacle. 
     In the first case, as $a<b$, we must have a vertex $w$ adjacent to $v$ with label $x$ satisfying $x>b$, which makes $v$ a non-pinnacle. Deleting edge $e$, means that $b$ remains a non-pinnacle as its corresponding vertex $v$ is still adjacent to vertex $w$ with label $x$. However, label $a$ may have now become a pinnacle. This implies that $|S|\leq|T|\leq |S|+1$, as desired.
     In the second case, we assume $b$ is a pinnacle. Then deleting edge $e$ continues to ensure that $b$ is a pinnacle. Once again label $a$ may have now become a pinnacle. This again implies that $|S|\leq|T|\leq |S|+1$, as desired.
 \end{proof}
 One might begin this study by considering the graph operations of subdividing an edge, edge contraction, and the join of two graphs. 
    
In Theorem \ref{thm:count pinnacle sets} we gave a formula for the number of pinnacle sets of size $k$ provided there is a unique bottom pinnacle set.
However, as we showed in Figure~\ref{fig:two bottom pinnacles} not all graphs have a unique bottom pinnacle set, and hence Theorem~\ref{thm:count pinnacle sets} does not hold. So the following questions remain:
\begin{question}
    What must be true of $G$ and $k$ such that $\Poset(G,k)$ has a unique bottom pinnacle set of size $k$?
\end{question}
Or more generally:
\begin{question}
    Fix $b\geq 1$. What must be true of $G$ and $k$ such that $\Poset(G,k)$ has precisely  $b$ bottom pinnacle sets of size $k$? In this case, how many distinct pinnacle sets of size $k$ does $G$ have?
\end{question}
One place to begin this study is to consider the bottom pinnacle set(s) of trees and determining if there is a unique bottom element. Another consideration is determining graph families that have many bottom pinnacle sets.  
We can also ask questions regarding the number of labelings with a given pinnacle set.
 \begin{question}
     Given a graph $G$ with $n$ vertices and a subset $S\subseteq[n]$, how many labelings of $G$ have pinnacles set $S$? 
 \end{question}

We conclude by answering this question for a cycle graph on $n$ vertices and when $S$ consists of the largest $k$ integers in $[n]$ . To formally state this result we recall that a \textit{composition} of an integer $s$ with $t$ parts is an ordered $t$-tuple  $(a_1,a_2,\ldots,a_t)$ with strictly positive integer entries whose coordinate sum equals $s$. Namely $a_i>0$ for all $i\in[t]$ and $\sum_{i=1}^t a_t=s$.
Let $\mathrm{comps}(s,t)$ denote the set of compositions of $s$ into $t$ parts. 

\begin{theorem}
    Let $k\leq \lfloor\frac{n}{2}\rfloor$ and let \(S=[n-k+1,n]\) be a size $k$ pinnacle set of the cycle graph \(C_n\). 
    Then the number of labelings of $C_n$ with pinnacle set $S$ is given by \[(k-1)!\cdot n\cdot 2^{n-2 k}\cdot\sum_{(a_1,a_2,\ldots,a_k)\in \mathrm{comps}(n-k,k)
    }\binom{n-k}{a_1,a_2,\dots,a_{k}}.\]
\end{theorem}
\begin{proof}
As in Figure~\ref{fig:labelingCn}, order the vertices $\{v_i\}_{i=1}^n$ of $C_n$ clockwise with $v_n$ placed at the top of the circle, and followed by $v_1$.
   We know that pinnacles cannot be adjacent, so we first select the position at which the pinnacles will be placed. To begin we position the pinnacle $n$, at the vertex $v_n$. 
   To know where the remainder of the pinnacles lie, it is equivalent to know the number of vertices between any two pinnacles. Since there are $k$ pinnacles, there are $k$ sets of vertices who lie completely between any two pinnacles. We call those sets of vertices lying between two pinnacles the set of gaps. 

   Moreover, since there are $k$ pinnacles, there are $n-k$ non-pinnacle values to place in those gaps. Each composition of $n-k$ with exactly $k$ parts gives all of the possible sizes of the gaps, and the way of selecting non-pinnacle numbers to place within the gaps is given by a multinomial. 
   Thus the number of ways of selecting the size of the gaps and selecting among the non-pinnacle numbers to fill those gaps with is given by 
   \[\sum_{(a_1,a_2,\ldots,a_k)\in\mathrm{comps}(n-k,k)}\binom{n-k}{a_1,a_2,\dots,a_{k}}.\]
   Now for each composition $(a_1,a_2,\ldots,a_k)$ of $n-k$ and the choice of what non-pinnacle numbers go into each gap, we must arrange these numbers so as to not create any pinnacles within the gaps. 
   For each gap we claim this can be done in $2^{a_i-1}$ ways. 
To see this, notice that since the $a_i$ numbers are all distinct, you can place the smallest number at the $j$th vertex on that path of vertices (where $1\leq j\leq a_i$). Then you must select $j-1$ numbers to place to the left and order them in decreasing order. This can be done in $\binom{a_i-1}{j-1}$ ways. 
The remaining numbers must be placed to the right of where we placed the smallest value and all placed in increasing order. There is only one way to do this. Now we must account for all of the possible places we can put the smallest number, hence the total number of ways to arrange the $a_i$ numbers so as to not create any pinnacles is given by 
\[\sum_{j=1}^{a_i}\binom{a_i-1}{j-1}=2^{a_i-1}\]
   as claimed.
   
   Once we know what numbers go in each gap, arranging the numbers in each gap so as to not create any pinnacles can be done independently. Hence, for any composition, the total number of ways to arrange the non-pinnacle values is given by \[\prod_{i=1}^k 2^{a_i-1}=2^{n-2k}.\]

   To complete the proof we note that we can permute the position of the pinnacles, not including the pinnacle $n$, which  we anchored at vertex $v_n$. This can then be done in $(k-1)!$ ways. 
   To conclude we note that we can rotate the labeling in order to account for all possible locations of where to place $n$ within the cycle graph. This can be done in $n$ ways. 
   Taking these two factors into account gives the final count for the total number of labelings of $C_n$ with pinnacle set $[n-k+1,n]$:
   \[(k-1)!\cdot n\cdot 2^{n-2 k}\cdot\sum_{(a_1,a_2,\ldots,a_k)\in\mathrm{comps}(n-k,k)}\binom{n-k}{a_1,a_2,\dots,a_{k}}.\qedhere\]
\end{proof}

\section*{Acknowledgements}
P.~E.~Harris was partially supported through a Karen Uhlenbeck EDGE Fellowship. S.~Lasinis was partially supported through a SURF award at the University Wisconsin, Milwaukee.

    \bibliographystyle{plain}
    \bibliography{refs}

\end{document}